\documentclass[12pt]{article}
\usepackage[a4paper,margin=1in]{geometry}
\usepackage{amsthm}

\usepackage{amssymb}
\usepackage{amsmath}
\usepackage{eufrak}

\newtheorem{theorem}{Theorem}
\newtheorem{defn}[theorem]{Definition}

\newtheorem{corollary}[theorem]{Corollary}
\newtheorem{lemma}[theorem]{Lemma}

\newtheorem{definition}[theorem]{Definition}

\newtheorem{proposition}[theorem]{Proposition}
\newtheorem{notation}[theorem]{Notation}
\newtheorem{prop}[theorem]{Proposition}
\begin{document}
\def\F{{\mathbb F}}
\title{On constructing deformations of noncommutative algebras}
\author{ Agata Smoktunowicz}
\date{ }
\maketitle
\begin{abstract}
We show that if there is a flat deformation from a finite-dimensional algebra N to an
algebra $A$, then such a deformation can be obtained by a slight generalization of the
construction from \cite{DDS}. The generalization allows negative powers of $t$ to occur in the
formula for $f$. The proof uses a result from
\cite{JJAS}. We also prove that this modified construction produces well-defined
formal deformations that are flat deformations from $N$ to $A$.
 \end{abstract}

\section{Introduction}
Deformation theory provides a meeting ground for researchers from various areas.
 In the noncommutative setting, 
deformations play an important role in representation theory, algebraic geometry, 
and mathematical physics, where one is often interested not only in the existence of a 
deformation, but also in constructing it explicitly. For finite-dimensional algebras, explicit constructions make it possible to pass from abstract existence questions to 
concrete formulas, to compare algebraic and geometric constructions, and to test 
general ideas on examples.
 The aim of this paper is to prove the following result:

\begin{theorem} \label{main} Let $n$ be a natural number. Let $N, A$ be finite-dimensional unital associative $\mathbb C$-algebras of dimension $n$. 
 The following statements are equivalent:
\begin{itemize}
\item There is a flat deformation from $N$ to $A$  (in the sense of Definition \ref{Deformation}).
\item  There exists a unital ${\mathbb C}\{t\}$-algebra homomorphism  $f$ satisfying the assumptions of Subsection \ref{333} and such that $Im(f)/t\cdot Im(f)$ is isomorphic as a unital  $\mathbb C$-algebra to $N.$ 
\end{itemize}
\end{theorem}
One source of motivation for the present paper comes from questions of Michael 
Wemyss concerning contraction algebras. These are finite-dimensional 
noncommutative algebras associated to flops, introduced by Donovan and Wemyss 
\cite{DW}, and closely related to maximal modification algebras developed by Iyama and 
Wemyss \cite{IW}.
 From the geometric side,  Wemyss asked whether a given contraction 
algebra can deform to a specific semisimple algebra arising naturally from the 
geometry. Toda constructed deformations of this kind, but it is not known in general 
whether these geometric deformations are flat \cite{Toda}. This makes it natural to seek 
algebraic methods that produce flat deformations explicitly and that can be applied in 
concrete examples. 
 In the follow-up paper we will use the methods developed here to construct deformations of some contraction algebras of type $E$.

These questions are closely tied to Gopakumar--Vafa invariants \cite{BW, GV, DW, K1, K2, 10, phdw}.  
 It is known that  when a 
contraction algebra deforms flatly to a unique semisimple algebra, the sizes and multiplicities of the matrix blocks  of that semisimple algebra recover the corresponding GV 
invariants \cite{W}. GV invariants are connected to Higgs fields in physics \cite{1, 30, 40, 20}.
 Deformations of Lie algebras appear in  
mathematical physics \cite{jose1, jose2, SWbook}.  Deformations of noncommutative algebras appear in connection with symplectic reflection algebras in \cite{IG}.  Some other interesting applications, in the context of algebraic geometry, can be found in \cite{JJ5, JJ, JJ3, JJ4}.

   In Subsection \ref{333} we present a method of constructing flat deformations of algebras. It is a generalization of the method from \cite{DDS} which allows one to treat situations that are inaccessible 
to the original construction. The final section shows that if there exists a flat deformation from a finite-dimensional algebra $N$ to an 
algebra $A$, then  there exists a flat deformation from $N$ to $A$ obtained by this generalized method. 
 This answers Question 0.1 of  \cite{DDS}  in the affirmative.
 Recall that this question, which appears in the introduction of \cite{DDS}, is related to the method of constructing  flat deformations introduced in \cite{DDS} and $f$ is defined as in Notation 1.7 of  \cite{DDS}:

$ $

{\bf Question 0.1.} \cite{DDS} Suppose that there is an associative and flat deformation of polynomial type from
a finite-dimensional algebra $N$ to a semisimple finite-dimensional algebra A. Does it follow
that there is a deformation from $N$ to $A$ obtained by this method, or by a generalization of this
method obtained by changing the function $f$?

  Observe that our Theorem \ref{main} answers this question in the affirmative. 
 We will now describe the modification to the method from \cite{DDS}. We extend the 
deformation procedure proposed there in two directions. First, we allow negative 
powers of the deformation parameter $t$ to appear in the defining formulas. Second, we  work with power-series rings instead of polynomials, which allows us to remove some assumptions from \cite{DDS}. 
 One advantage of our construction is that, to construct flat deformations of noncommutative algebras, 
 it is
enough to know defining relations
and the dimension, provided one
verifies the hypotheses of the
method, for instance
 $f(s_{i}) \in t\cdot Im(f)$ and the required
dimension bound.
  Thus,  it is not necessary to construct Groebner bases or give elements which span this algebra as a linear space. 
 This is illustrated in the follow-up paper  for some contraction algebras $E_{4,n}$ that are Jacobi algebras generated by only two relations. 
  
 Our construction does not use 
Hochschild cohomology and derivations. Those tools are fundamental in deformation 
theory, but for the purpose of building explicit examples it is often useful to have a 
more elementary and computational method, especially in the finite-dimensional 
noncommutative setting.
 
The methods developed here are intended as concrete tools for 
constructing and analyzing examples, and for making the passage from a prescribed 
target algebra to an explicit deformation as transparent as possible. In particular, they allow us
 to replace the assumption that there is a flat deformation from an algebra $N$ to $A$ by
 the existence of a function $f$ satisfying the required properties. 
 Notice that in the context of modules, exact sequences were used to describe deformations and degenerations of modules \cite{Y, Zwara}  (in a way not related to our paper). 

 The paper is organized as follows. In Section $2$ we collect the background on 
formal deformations and the notation used throughout. Sections \ref{21}-\ref{Y} develop the first 
method in detail, prove that it yields flat deformations, and analyze the resulting specialized algebras.  The later sections  show, using \cite{JJAS}, that if there exists a flat deformation from a finite-dimensional ${\mathbb C}$-algebra $N$ to an 
algebra $A$, then  there exists a flat deformation from $N$ to $A$ obtained by this method.  The last section formulates a generalized method  
 which is easier to use in practice. Throughout, all algebras are associative and usually noncommutative.  Moreover,  all considered algebras are unital, and all algebra homomorphisms considered  are homomorphisms of unital algebras, unless stated otherwise. 
\section{Background information}
 Some of the  following background information is taken from \cite{Wemyss2}, since  we  use the same notation.
 We recall the definition of a formal deformation of $A$ from \cite{SWbook}.
 
\begin{defn}\label{SWbook} A formal deformation $(A\{t\}, \circ, +)$ of a $\mathbb C$-algebra $(A, \cdot, +)$ is an associative ${\mathbb C}\{t\}$-bilinear multiplication $\circ $ on the $\mathbb C\{t\}$-module $A\{t\}$ (where $A\{t\}$ is the set of power series with coefficients from $A$). Moreover, $(A\{t\}, \circ , +)/t(A\{t\}, \circ , +) \cong  (A, \cdot , +)$ as a
 $\mathbb C$-algebra.

Recall that $A\{t\}$ has a natural structure of a $\mathbb C\{t\}$-module given by 
\[t^{j}\sum_{i=0}^{\infty }a_{i}t^{i}= \sum_{i=0}^{\infty }a_{i}t^{i+j}.\]
\end{defn}

 $(A\{t\}, \circ, +)$ is a deformation of {\em polynomial type} of a $\mathbb C$-algebra $(A, \cdot, +)$ if for all $a,b\in A$, $a\circ b\in A[t]$, where $A[t]$ denotes the polynomial ring in the variable $t$ with coefficients in $A$.  

 Let $d_{1}, \ldots , d_{n}\in A$ be a basis of $A$ as a vector space over $\mathbb C$. Note that every element of $A\{t\}$ can be written in the form $\sum_{i=1}^{n}d_{i}f_{i}(t),$ for some $f_{i}\in \mathbb C\{t\}$. We then have 
\[d_{i}\circ d_{j}=\sum_{k=1}^{n}g_{i,j,k}(t)d_{k}.\] 
 
Then we have a well-defined multiplication on $A\{t\}$ given by
\[(\sum_{i=1}^{n}d_{i}f_{i}(t))\circ (\sum_{j=1}^{n}d_{j}f_{j}(t))=\sum_{i=1}^{n}\sum_{j=1}^{n}d_{i}\circ d_{j}f_{i}(t)f_{j}(t).\]

 This multiplication is associative
because the structure constants satisfy the
associativity relations. Since $ d_{1}, \ldots , d_{n}$ is a
basis, these relations determine an associative algebra structure on $A\{t\}$. If $(A\{t\}, \circ , +)$  has an identity element $1_{A}$, we can identify elements of $\mathbb C\{t\}$ with elements of the subring  ${\mathbb C}\{t\}\cdot 1_{A}\subseteq A\{t\}$, and we have $t^{i}\circ a=at^{i}=a\circ t^{i}$ for all $i$ and all $a\in A$. In particular $t$ is not a zero divisor in the algebra $(A\{t\}, \circ, +)$. 
 Similarly, if $0\neq f(t)\in \mathbb C\{t\}$, then $f(t) $ (which we identify with $f(t)\cdot 1_{A}$) is not a zero divisor of $A\{t\}$.

 We recall a definition of a flat deformation used in algebraic geometry:
\begin{definition}\label{Deformation} Let $N$ be a unital $\mathbb C$-algebra. A deformation of  $N$  over ${\mathbb C}\{t\}$ is a ${\mathbb C}\{t\}$-algebra  $\mathcal{N}$ such that $ \mathcal{N}$ is a free ${\mathbb C}\{t\}$-module and  the algebra $\mathcal{N} / t \mathcal{N}$ is isomorphic to $ N.$
 We assume that the image of the identity element of $N$ is the identity element in $\mathcal N$.
 A flat deformation of the algebra $N$ to a unital  algebra $A$ over ${\mathbb C}\{t\}$  is a deformation over ${\mathbb C}\{t\}$ as above  such that $ \mathcal{N} \otimes_{{\mathbb C}\{t\}} {\mathbb C}\{\{t\}\}$ is isomorphic  to $ A\{\{t\}\}=A\otimes _{{\mathbb C}} {\mathbb C}\{\{t\}\}$,  where ${\mathbb C}\{\{t\}\}$ denotes the algebraic closure of the fraction field of ${\mathbb C}\{t\}$ (so ${\mathbb C}\{\{t\}\}$ is an algebraically closed field), and $\otimes $ denotes tensor product over the indicated central subalgebra. 
\end{definition} 
 \subsection{Notation}
 All algebras considered in this paper have an identity element, so they are unital algebras, unless stated otherwise. Moreover,  all considered homomorphisms of algebras are homomorphisms of unital algebras, unless stated otherwise. 

\begin{notation}\label{4444}   By $\mathbb C\{t\}$ we denote the power-series ring in the  variable $t$ with coefficients in $\mathbb C$ (this notation is typical in noncommutative ring theory, and we don't assume that these power series are convergent;  note however that in some algebraic geometry papers the notation ${\mathbb C}[[t]]$ is used instead). By ${\mathbb C}[t]$ we denote the polynomial ring in the variable $t$ (so $t$ is a central element in these rings).

    We will denote by $\mathbb C\langle x_{1}, \ldots , x_{\rho }\rangle\{t\}$ the free unital  ${\mathbb C}\{t\}$-algebra on the free generators $ x_{1}, \ldots , x_{\rho }$.  So ${\mathbb C}\langle x_{1}, \ldots , x_{\rho}\rangle \{t\}$ (which can also be denoted as  $ {\mathbb C}\{t\}\langle x_{1}, \ldots , x_{\rho}\rangle $) is 
the ordinary free associative
algebra over ${\mathbb C}\{t\}$, consisting
of finite ${\mathbb C}\{t\}$-linear
combinations of words.

  By $\mathbb C\langle x_{1}, \ldots , x_{\rho } \rangle [t]$ we will  denote the polynomial ring in the  variable $t$ over the ring  $\mathbb C\langle x_{1}, \ldots , x_{\rho } \rangle $.
\end{notation}

\begin{notation}\label{not:monomials}
   Let $ \rho$ be a natural number.  Consider monomials of the ring $\mathbb C\langle x_{1}, \ldots , x_{ \rho } \rangle$. The monomials are words in the letters $x_{1}, \ldots , x_{\rho }$. We can order the monomials using shortlex ordering by specifying $x_{1}<\ldots <x_{\rho }$ and setting $1$ as the least monomial. Note that this means that if a monomial $p_i\in \mathbb C\langle x_{1}, \ldots , x_{\rho } \rangle$ is a product of $n_i$ elements from the set $\{x_{1}, \ldots , x_{\rho }\}$, then $p_i<p_j$ whenever $n_i<n_j$. If $n_i=n_j$, then  write $p_i=a_1a_2...a_{n_i}$ and $p_j=b_1b_2...b_{n_j}$, where $a_k, b_k\in\{x_{1}, \ldots , x_{\rho }\}$. Then $p_i<p_j$ if and only if $a_k<b_k$, where $k$ is the first index where $a_k$ and $b_k$ differ.  We denote the monomials of $\mathbb C\langle x_{1}, \ldots , x_{\rho } \rangle$ by $ p_1,p_2, \dots$
where $p_i<p_{j}$ for $i<j$. 
\end{notation}

\begin{notation}
   Let $R$ be a ring and let $I$ be an ideal of $R$. Then elements of the factor ring $R/I$ will be denoted as $r+I$, where $r\in R$.
 Notice that $r+I=s+I$ for $r,s\in R$ if and only if $r-s\in I$.
\end{notation}

\begin{definition}\label{1} Let $A$ be a finite-dimensional algebra over $\mathbb C$ of dimension $n$. 
 Let $A((t))$ denote the algebra of Laurent series over $A$ in the variable $t$, so
 elements of $A((t))$ are linear combinations of elements $at^{i}g(t)$ where 
$i$ is  an integer, $g(t)\in {\mathbb C}\{t\}$ and $a\in A$.

 We also define $A[t,t^{-1}]$ to be a subalgebra of $A((t))$ consisting of Laurent polynomials in the variable $t$, so elements of $A[t, t^{-1}]$ are   linear combinations of elements $at^{i}g(t)$ where 
$i$ is  an integer, $g(t)\in {\mathbb C}[t]$, and $a\in A$. 
\end{definition}

\section{ The method}\label{21}

\subsection { Describing the method}\label{333}

 We first describe the assumptions of our method. This method is obtained by slightly modifying the definition of the function $f$ in the method from \cite{DDS}. 

 All algebras considered in this section are associative and unital, and usually noncommutative.  Moreover,  all algebra homomorphisms considered in this section are homomorphisms of unital algebras. 
\begin{enumerate}
 \item  Let notation  be as in Definition \ref{1}.  Let $A$ be a finite-dimensional unital algebra over $\mathbb C$ of dimension $n$.  Let $\rho$ be a natural number and let 
 $f: {\mathbb C}\langle x_{1}, \ldots , x_{\rho}\rangle \{t\} \rightarrow A((t))$ be a   homomorphism of  ${\mathbb C}\{t\}$-algebras such that  for every $i $ \[f(x_{i})\in A[t,t^{-1}].\]  
  Notice that $f$ is a homomorphism of unital algebras, so \[f(1)=1_{A((t))}.\] 
 
\item  We assume that there are $c_{1}, \ldots , c_{n}\in {\mathbb C}\langle x_{1}, \ldots , x_{\rho }\rangle $ such that $f(c_{1}), \ldots , f(c_{n})$ are 
linearly independent over ${\mathbb C}\{t\}$.  
\item We also assume that there is a natural number $\gamma $ such that
 \[\operatorname {Im} (f)\subseteq t^{-\gamma}A\{t\}.\] 
\end{enumerate}

{\em If the above properties $1-3$ are satisfied, then there is a flat deformation from the ${\mathbb C}$-algebra $\operatorname {Im} (f)/t\cdot \operatorname {Im} (f)$ to $A$ (see Corollary \ref{pppp})}.

$ $
 
We will now describe how the method works. The general idea is similar to the one in \cite{DDS}. 

$ $

\begin{theorem} \label{9}
 Assume that $f$ satisfies the hypotheses of Subsection \ref{333}.
Then there exist  $q_{1}, \ldots ,  q_{n } \in  {\mathbb C}\langle x_{1}, \cdots , x_{\rho }\rangle [t]$ such that \[\operatorname {Im} (f)= \sum_{i=1}^{n}f(q_{i}){\mathbb C}\{t\},\] 
and the elements $f(q_{i})$ are linearly independent over ${\mathbb C}\{t\}$. Moreover, for all $1\leq k, m\leq n$ 
one has 
   \[f( q_{k})f(q_{m})=\sum_{i=1}^{n}(\zeta _{i,k,m}+t\xi_{i,k,m}(t))f(q_{i}), \]
 with $\zeta _{i,k,m}\in \mathbb C$, $\xi _{i,k,m}(t)\in {\mathbb C}\{t\}$.  These constants define a finite-dimensional
algebra $N$ with basis $d_{1}, \ldots , d_{n}$ by 
\[d_{k} d_{m}= \sum_{i=1}^{n}\zeta _{i,k,m}d_{i},\]
 and a formal deformation $\mathcal N$ over ${\mathbb C}\{t\}$ by 
\[d_{k}\ast _{t} d_{m}= \sum_{i=1}^{n}(\zeta _{i,k,m}+t\xi_{i,k,m}(t))d_{i}.\]
Then $N\cong \operatorname {Im} (f)/t \cdot \operatorname {Im} (f)$ and the ${\mathbb C}\{t\}$-algebra $\mathcal N\cong \operatorname {Im} (f)$ is a flat deformation from $N$ to $A$.
\end{theorem}
\begin{proof} The existence and spanning properties of $q_{i}$ follow from Proposition \ref{basisnew}; the multiplication formula follows from Proposition \ref{prop:relnew}; the construction of the formal deformation follows from Theorem  \ref{thm formal deformationnew}; the identification of the special fibre is Proposition \ref{imp}; and the flatness follows from Theorem   \ref{thm:iso1new}. $\mathcal N\cong \operatorname {Im} (f)$ is Proposition \ref{Im}.
  $ N\cong \operatorname {Im} (f)/t\cdot \operatorname {Im} (f)$ is Proposition \ref{imp}.
\end{proof}

\subsection{Examples} \label{examples}
 The following examples are written in the notation of the paper. In each case, we choose
a finite-dimensional target algebra $A$
over $\mathbb C$, define a homomorphism
\[ f : {\mathbb C}\langle x_{1}, \ldots , x_{\rho }\rangle \{t\} \rightarrow A((t)),\] 
and compute the special fibre
\[N \cong  \operatorname {Im} (f)/t\cdot \operatorname {Im} (f).\]

Let $n$ be the dimension of $A$. 
We  find $n$ elements $m_{1}, \ldots , m_{n}\in {\mathbb C}\langle x_{1}, \ldots , x_{\rho }\rangle $ such that $f(m_{1}), \ldots , f(m_{n})$ are linearly independent over ${\mathbb C}\{t\}$. Thus, by the main construction of the paper, the algebra $N$ 
flatly deforms to $A$.

The examples are deliberately elementary. They are meant to illustrate Theorem \ref{9}. 
In the example below, we calculate the concrete deformation and elements $q_{1}, \ldots , q_{n}$.

 {\bf Example } (The dual numbers deforming to ${\mathbb C}\oplus {\mathbb C}$). Let 
 \[A={\mathbb C}e_{1}\oplus {\mathbb C}e_{2}, e_{i}^{2}=e_{i}, e_{1}e_{2}=0, 1_{A}=e_{1}+e_{2}.\] 
 Define 
\[ f : {\mathbb C}\langle x\rangle \{t\} \rightarrow A((t)), f(x) = t\cdot e_{2}.\]
 Then
\[ \operatorname {Im} (f) = {\mathbb C}\{t\} \cdot  1_{A((t))} + {\mathbb C}\{t\} \cdot  y, y:= f(x) = t\cdot e_{2},\]
and
\[y^{2} = t^{2}\cdot e_{2} = t\cdot y.\]
 Set $q_{1}=1, q_{2}=x$. Then $f(1)=1_{A((t))}$, $f(x)=y$.  
Hence the corresponding deformation algebra has basis $1_{A},  y$ over ${\mathbb C}\{t\}$ and relation
\[ y^{2} = t\cdot y.\]
Modulo $t$, this becomes
\[y^{2} = 0.\]
 Therefore 
\[\operatorname {Im} (f)/t\cdot \operatorname {Im} (f)\cong {\mathbb C}[e]/\langle e^{2}\rangle .\]
 Note that $A={\mathbb C}\oplus {\mathbb C}$ has dimension $2$, and the elements $f(1), f(x)$ are linearly independent over ${\mathbb C}\{t\}$. 
This gives a flat deformation from $N= {\mathbb C}[e]/\langle e^{2}\rangle $ to $A= {\mathbb C}\oplus {\mathbb C}$.

The next example is a compact noncommutative example. 
  This example shows  why negative powers of $t$ are useful. 
 In this example, we will not  compute elements $q_{i}$, because this is not necessary for applying 
our method. We will only show that a flat deformation from $N$ to $A$ exists without giving a concrete deformation.  

$ $

{\bf Example.} 
Let $A=M_2({\mathbb C})$ and denote by $E_{ij}$ the usual matrix units.  Define
\[
        f:{\mathbb C}\langle x,y\rangle\{t\}\longrightarrow M_2({\mathbb C})((t))
\]
by
\[
        f(x)=t^{2}E_{12},
        \qquad
        f(y)=t^{-1}E_{21}.
\]
  Observe that $\operatorname {Im} (f)$ has four linearly independent elements over ${\mathbb C}\{t\}$:
\[ f(x), f(y), f(xy), f(1),\] since $f(1)=E_{11}+E_{22}$ is the identity matrix and $f(xy)=tE_{11}$.
 This example uses a negative power of $t$, so it is not literally covered by the original
polynomial version of the construction. It is, however, covered by the modified Laurent
version used in the paper, because
\[ \operatorname {Im} (f) \subseteq  t^{-1}M_{2}({\mathbb C})\{t\}.\]

We calculate the algebra $\operatorname {Im} (f)/t\cdot \operatorname {Im} (f)$. 
 We use Lemma \ref{condition2} (proved below)  applied with  $\rho =2$ (since we have two generators 
$x$ and $y$). Recall that $n$ is the dimension of $A$ and in our case  $A=M_{2}({\mathbb C})$ has dimension $4$, so $n=4$. 
 Let $s_{1}=x^{2}, s_{2}=y^{2}, s_{3}=xy+yx$. Then \[f(s_{1})=0, f(s_{2})=0, f(s_{3})=tI=t\cdot f(1),\] so \[f(s_{1}), f(s_{2}), f(s_{3})\in t\cdot \operatorname {Im} (f).\] Let $S$ be the ideal generated by $s_{1}, s_{2}, s_{3}$. Then 
 $ {\mathbb C}\langle x, y\rangle /S$ has dimension not exceeding $4$.
By Lemma  \ref{condition2}, 
 $ {\mathbb C}\langle x, y\rangle /S$ is isomorphic to $N$.
 Hence our method gives  a flat deformation from 
 $ {\mathbb C}\langle x, y\rangle /S$ to $M_{2}({\mathbb C})$.

\section{ Elements $q_{1},\ldots , q_{n}$}\label{A}

 In this section we generalize some results from \cite{DDS}. 
 In this section  we denote \[M=A, \tilde A=A[t,t^{-1}], A'=A((t)).\]
 
 We use the notation $M, {\tilde A}, A'$,  rather than $A, A[t, t^{-1}], A((t))$, simply because in Section \ref{generalisations} we will present a generalization of the method  from Subsection \ref{333} which uses a different $M, {\tilde A}$ and $A'$.

$ $

\begin{prop}\label{basisnew}  
    Let notation and assumptions be as in Subsection \ref{333}.  Assume that $\dim _{\mathbb C} M =\dim_{\mathbb C}A= n$. Then the following holds:
\begin{enumerate}
\item  There exist elements 
$q_{1},  \dots , q_{n}\in \mathbb C\langle x_{1}, \ldots , x_{\rho } \rangle [t]$ such that for every  
$r\in \mathbb C\langle x_{1}, \ldots , x_{\rho}\rangle \{t\}$ we have 
      \[  f(r)\in \sum_{i=1}^{n}\mathbb C\{t\} f(q_{i}).\] 
\item  There exists a natural number $m$ such that $rt^{m}\in \operatorname {Im} (f)$ for every $r\in M$.
\item There is $1\leq j\leq n$ such that $q_{j}=1$, where $1$ is the identity element of $\mathbb C\langle x_{1}, \ldots , x_{\rho}\rangle \{t\}$. 
\end{enumerate}
 
\end{prop}
\begin{proof} Let $k_1$ be the smallest possible integer such that for some nonzero $e_{1}\in M$ we have
\[e_{1}t^{k_1}+\sum_{i=k_{1}+1}^{\infty }\alpha_{i,1}\cdot t^{i} \in \operatorname {Im} (f)\]
for some $\alpha_{i,1}\in M$. Let $k_2$ be the smallest possible integer such that for some  $e_{2}\in M$, $e_{2}\notin \mathbb C \cdot e_{1}$ we have 
\[e_{2}t^{k_{2}}+\sum_{i=k_2+1}^{\infty }\alpha _{i,2} \cdot t^{i} \in \operatorname {Im} (f)\]
for some $\alpha_{i, 2}\in M$.

 Notice also that by the minimality of $k_{1}$ we have $k_{2}\geq k_{1}$.

Let $k_{3}$ be the smallest possible integer such that 
 for some $e_{3}\in M$, $e_{3}\notin \mathbb C \cdot e_{1}+{\mathbb C}\cdot e_{2}$ we have 
\[e_{3}t^{k_{3}}+\sum_{i=k_3+1}^{\infty }\alpha _{i,3} \cdot t^{i} \in \operatorname {Im} (f)\]
for some $\alpha _{i,3}\in M$. Notice also that by the minimality of $k_{1}$ and $k_{2}$ we have $k_{3}\geq k_{2}$.

 Continuing until no further elements $k_{i}, e_{i}$ can be added, we define 
 integers $k_{1}, \ldots , k_{n'}$ such that  $k_{1}\leq k_{2}\leq \ldots \leq k_{n'}$. We have also defined the corresponding elements 
 $e_{1}, \ldots , e_{n'}\in M$. Since $M$ has dimension $n$ as a  $\mathbb C$-vector space,  $n'\leq n$.

 There are  $q_{1}, q_{2}, \cdots , q_{n'}\in \mathbb C\langle x_{1}, \ldots , x_{\rho  }\rangle \{t\}$ such that 
\[f(q_{1})=e_{1}t^{k_1}+\sum_{i=k_{1}+1}^{\infty }\alpha_{i,1}\cdot t^{i}\]
\[f(q_{2})=e_{2}t^{k_2}+\sum_{i=k_{2}+1}^{\infty }\alpha_{i,2}\cdot t^{i}\]
 and so on.

We next show that $n'=n$. Since $M$ has dimension $n$ as a vector space over $\mathbb C$,  $n'\leq n$. 
 If $n'<n$, we can find elements $e_{n'+1}, \ldots , e_{n}\in M$ such that $e_{1}, \ldots , e_{n}$ span $M$ as a $\mathbb C$-linear space.
 By assumption $2$ from Subsection \ref{333} there are  $m_{1}, \ldots , m_{n}\in {\mathbb C}\langle x_{1}, \ldots , x_{\rho }\rangle $ such that $f(m_{1}), \ldots , f(m_{n})$ are linearly independent over ${\mathbb C}\{t\}$. Observe also that $\operatorname {Im} (f)\subseteq A((t))=e_{1}{\mathbb C}((t))+\cdots +e_{n}{\mathbb C}((t))$. Therefore,   for each $i\leq n$, $f(m_{i})=\sum_{j=1}^{n} a_{i,j}(t)e_{j}$ for some $a_{i,j}(t)\in {\mathbb C}((t))$. Recall that  the elements $f(m_{i})$ are linearly independent over ${\mathbb C}\{t\}$. It follows that the determinant of the $n$-by-$n$ matrix $A$ whose $i,j$-th entry is $a_{i,j}(t)$ is nonzero. Therefore, $A$ is an invertible matrix in $M_{n}({\mathbb C}((t)))$. It follows that  for each $e_{i}$ there is $g(t)\in {\mathbb C}\{t\}$ such that $e_{i}g(t)\in \operatorname {Im} (f)$. 

 Consequently,  $e_{i}(t^{m'}(1+h(t)))\in \operatorname {Im} (f)$ for sufficiently large $m'$ and for some $h(t)\in t{\mathbb C}\{t\}$. 
 This implies that $e_{i}t^{m'}$ is the leading term of an element from $\operatorname {Im} (f)$. 
Thus $n=n'$; otherwise, we could add
$e_{n'+1}$ to the list, contradicting maximality 
of the construction. 
 Observe also that $e_{i}t^{m'}(1+h(t))^{-1}\in \operatorname {Im} (f)$ and $(1+h(t))^{-1}\in \mathbb C \{t\}$.  
 Therefore, for sufficiently large $m'$ and  for  every $a\in M$  
 \[at^{m'}\in \operatorname {Im} (f).\] Consequently 
$at^{m'+j}\in t^{j}\operatorname {Im} (f)$, hence for any $r\in M\{t\}$ we have 
$t^{m'}r\in \operatorname {Im} (f)$. We can take $m'$ such that $k_{1}, \ldots , k_{n}<m'$.  

$ $

We will now show that \[\operatorname {Im} (f)\subseteq \mathbb C\{t\}f(q_{1})+\cdots +\mathbb C\{t\}f(q_{n}).\]

 Recall that \[f(q_{1})=e_{1}t^{k_1}+\sum_{i=k_{1}+1}^{\infty }\alpha_{i,1}\cdot t^{i}.\] Each $\alpha _{i,1}$ can be written as a linear combination over $\mathbb C$ of $e_{1}, \ldots , e_{n}$. This  implies
   \[f(q_{1}(1+t\cdot h_{1}(t)))=e_{1}t^{k_1}+\sum_{i=k_{1}+1}^{\infty }\alpha_{i,1}'\cdot t^{i}\] with $\alpha _{i,1}'\in {\mathbb C}e_{2}+\ldots +{\mathbb C}e_{n},$
 for some $h_{1}(t)\in {\mathbb C}\{t\}$.

 Similarly, by the minimality of $k_{1}$ for each $j\leq n$ we have   \[f(q_{j}- q_{1}t\cdot h_{j}(t))=e_{j}t^{k_j}+\sum_{i=k_{j}+1}^{\infty }\alpha_{i,j}'\cdot t^{i}\] with $\alpha _{i,j}'\in {\mathbb C}e_{2}+\ldots +{\mathbb C}e_{n},$
 for some $h_{j}(t)\in {\mathbb C}\{t\}$.
 Denote ${\bar q}_{j}=q_{j}-q_{1}t\cdot h_{j}(t)$. We can now apply the same reasoning to elements $f({\bar q}_{j})$ for $j\geq 2$, by using the minimality of $k_{2}$. We obtain that there are ${\bar h}_{j}(t)\in {\mathbb C}\{t\}$ such that 
${\bar q}_{j}-{\bar q}_{2}t\cdot {\bar h}_{j}(t),$ for $j\geq 2$, satisfies \[f({\bar q}_{j}- {\bar q}_{2}t\cdot {\bar h}_{j}(t))=e_{j}t^{k_j}+\sum_{i=k_{j}+1}^{\infty }\alpha_{i,j}''\cdot t^{i}\] with 
$\alpha _{i,j}''\in {\mathbb C}e_{3}+\ldots +{\mathbb C}e_{n}$. 
 Continuing in this way we obtain that for each $j\leq n$ there are $q_{j}'\in q_{j}+t\sum_{1\leq i\leq j}{\mathbb C}\{t\}q_{i}$
  such that \[f(q_{j}')=  e_{j}t^{k_j}+\sum_{i=k_{j}+1}^{\infty }\alpha_{i,j}'''\cdot t^{i}\] with $\alpha _{i,j}'''\in {\mathbb C}e_{j+1}+\ldots +{\mathbb C}e_{n}.$
 
Let $c\in  \mathbb C\langle x_{1}, \ldots , x_{\rho }\rangle $. We will show that
 \[f(c)\in  \mathbb C\{t\}f(q_{1})+\cdots +\mathbb C\{t\}f(q_{n}).\] 
  By the definition of $q_{j}'$ it suffices to show that \[f(c)\in \mathbb C\{t\}f(q_{1}')+\cdots +\mathbb C\{t\}f(q_{n}').\]
 Let \[f(c)=\sum_{i=1}^{n} e_{i}t^{\gamma _{i}}h_{i}(t),\] for some integers $\gamma _{i}$, some $h_{i}(t)\in {\mathbb C}\{t\}$ such that if $h_{i}(t)\neq 0$ then $h_{i}(t)\notin t {\mathbb C}\{t\}$. By the minimality of $k_{1}$ we get 
 $k_{1}\leq \gamma _{1}$ (if $h_{1}(t)=0$ we can also make this assumption). 
 Let $c_{2}=c- q_{1}'h_{1}(t)t^{\gamma _{1}-k_{1}}$. Then,  
\[f(c_{2})=\sum_{i=2}^{n}t^{\beta _{i}}g_{i}(t)e_{i}\] 
for some integers $\beta _{i}$, some $g_{i}(t)\in {\mathbb C}\{t\}$ such that if $g_{i}(t)\neq 0$ then $g_{i}(t)\notin t {\mathbb C}\{t\}$. By the minimality of $k_{2}$ we get 
 $k_{2}\leq \beta _{2}$ (if $g_{2}(t)=0$ we can also make this assumption). Repeating the same argument for the 
 indices $2, \ldots ,n$, we obtain that 
$f(c_{n})=0$ where $c_{n}\in c+\sum_{i=1}^{n}{\mathbb C}\{t\}q_{i}'$. Therefore,
$f(c)\in \sum_{i=1}^{n}{\mathbb C}\{t\}f(q_{i}')$, and hence $f(c)\in \sum_{i=1}^{n}{\mathbb C}\{t\}f(q_{i}).$

$ $

 We will now show that we can assume that $q_{i}\in \mathbb C\langle x_{1}, \ldots , x_{\rho }\rangle [t] $, so that they are polynomials in the variable $t$.

 Let $q_{i}=\sum_{l=0}^{\infty } d_{l}t^{l}$ for some $d_{l}\in \mathbb C\langle x_{1}, \ldots , x_{\rho }\rangle $.   
 Denote $q_{i}'=\sum_{l=0}^{m'+\gamma +1}d_{l}t^{l}$.
 Then  $f(q_{i}')$ and $f(q_{i})$  have the same leading term for each $i$ (since leading terms appear near $t^{k_{i}}$ and $k_{i}<m'$ and also  since by  assumption $3$  from Subsection \ref{333}  $f(d_{l})\in t^{-\gamma }M\{t\}$).
 We can therefore use the same reasoning from the first part of this proof, choosing $q_{i}'$ instead of $q_{i}$ at each step  (so choosing $\alpha _{i,j}$ such that $f(q_{i}')=e_{i}t^{k_{i}}+\sum_{l=k_{i}+1}^{\infty} {\alpha }_{l,j}t^{l}$). Notice that at the $i$-th step, it is not possible to obtain powers of $t$ smaller  than $k_{i}$; this can be shown by looking at the first $i$ for which it happens and observing that if a smaller $k_{i}$ appears by adding some element $q_{i}''$ to our list instead of $q_{i}'$, then the same element added to the original list of $q_{1}, \ldots ,q_{i-1}$ at  the $i$-th step would also give a smaller $k_{i}$ (since $q_{j}$ and  $q_{j}'$ have the same leading coefficients for each $j$),  contradicting the minimality of the original element 
$k_{i}$ and the choice of  $q_{i}$.

  In the first part of this proof it was shown that $\operatorname {Im} (f)\subseteq \mathbb C\{t\}f(q_{1})+\cdots +\mathbb C\{t\}f(q_{n})$. 
 We can  use the same reasoning   for elements $q_{i}'$ instead of elements $q_{i}$ (for $1\leq i\leq n$) and obtain  
 $\operatorname {Im} (f)\subseteq \mathbb C\{t\}f(q_{1}')+\cdots +\mathbb C\{t\}f(q_{n}')$. 

 Therefore we can take elements $q_{1}', \ldots , q_{n}'$ instead of 
 $q_{1}, \ldots , q_{n}$. So we can assume that $q_{1}, \ldots , q_{n}\in \mathbb C\langle x_{1}, \ldots , x_{\rho }\rangle [t]$.

We will now show that we can assume that $q_{j}=1$ for some $j$. 
  Suppose that the pairs $(e_{i}, k_{i})$ with $k_{i}<0$ have been chosen for $i<j$, and that the next pair $(e_{j}, k_{j})$ satisfies $k_{j}\geq 0$. Such $j$ exists, as otherwise the span of $e_{1}, \ldots , e_{n}$ would contain  $1_{A}$ (which is impossible    since we show below  that $1\notin {\mathbb C}e_{1}+\ldots +{\mathbb C}e_{j-1}$ for $j=n+1$). Since $f(1) = 1_{A'}$, 
if $1 \notin {\mathbb C}e_{1} + \cdots  + {\mathbb C}e_{j-1}$, the
minimal possible exponent  for
the next vector is at most $0$;
combined with $k_{j} \geq 0$, this
gives $k_{j} = 0$, so $q_{j} = 1$ may be
chosen (and $e_{j}=1$). It remains to show that $1\notin {\mathbb C}e_{1}+\ldots +{\mathbb C}e_{j-1}$ (where $j\leq n+1$).  Suppose on the contrary that 
$1\in {\mathbb C}e_{1}+\ldots +{\mathbb C}e_{j-1}$.  By the assumption about the choice of $e_{i}, k_{i}$ for $i<j$ (after multiplying by an appropriate power of $t$) we have for each $1\leq i<j$:
 \[e_{i}\in t(A\{t\}+\operatorname {Im} (f)).\]
 Therefore, $1+tc\in t\operatorname {Im} (f)$ for some $c\in A\{t\}$. Since $f$ is a homomorphism of algebras,  
 $\operatorname {Im} (f)\cdot \operatorname {Im} (f)\subseteq \operatorname {Im} (f).$ 
Let $\gamma $ be as in Subsection \ref{333}. We obtain $(1+tc)^{\gamma +1}\in t^{\gamma +1}\operatorname {Im} (f)$, and so $t^{-\gamma -1} (1+tc)^{\gamma +1}=f(d)$ for some $d\in \mathbb C\langle x_{1}, \ldots , x_{\rho }\rangle \{t\}$.   
  Since $(1+tc)^{\gamma +1}\neq 0$ it follows that $f(d)\in t^{-\gamma -1}A\{t\}$ and $f(d)\notin  t^{-\gamma }A\{t\}$, contradicting Assumption $3$ from  Subsection \ref{333}. We have obtained a contradiction. This proves  that we may assume that $q_{j}=1$ for some $j\leq n$.  
\end{proof}

\section{Some properties of elements $q_{i}$}\label{B}
The following proposition is similar to Proposition $3.1$ from \cite{DDS} and has a similar proof.  Recall that elements $p_{i}$ were defined in Notation \ref{not:monomials}. 
\begin{prop}\label{basesnew}
Let notation and assumptions be as in Subsection \ref{333}.  Let $p_{k}$ be as in Notation \ref{not:monomials}. 
 Denote $ker(f)=I$.
 Then the following holds:

\begin{enumerate}
\item   Suppose $\sum_{i=1} ^n \beta _{i}q_{i}\in I$ for some $\beta _{i}\in {\mathbb C}\{t\}$. Then $f(\sum_{i=1} ^n \beta _{i}q_{i})=0$ and consequently  $\beta _{i}=0$ for every $i\in\{1,2,\dots, n\}$.

\item For each $k\in \mathbb N$ there are 
 $\zeta _{i,k}\in \mathbb C$, $\xi _{i,k}(t)\in {\mathbb C}\{t\}$  such that 
\[p_{k}-\sum_{i=1}^{n}(\zeta _{i,k}+t\xi_{i,k}(t))q_{i}\in I.\]
 \item Moreover, all elements of $I$ are $\mathbb C \{t\}$-linear combinations
 of the elements 
\[p_{k}-\sum_{i=1}^{n}(\zeta _{i,k}+t\xi_{i,k}(t))q_{i},\] for some $k$. 
\end{enumerate}
\end{prop}
\begin{proof}
Suppose $\sum_{i=1}^n\beta_i q_i \in I$. Recall that $I= \ker f$, so
\[0 = f(\sum_{i=1}^n\beta_i q_i) = \sum_{i=1}^n\beta_i f(q_i). \]
 Let notation be as in the proof of Proposition \ref{basisnew}. Then 
\[f(q_m) = e_{m}t^{k_m}+\sum_{i=k_{m}+1}^{\infty }\alpha_{i,m}\cdot  t^{i}\]
for each $m\in \{1, \ldots , n\}$. Now suppose that not all $\beta_i$ are equal  to $0$. 
 Let $\xi_{i}t^{j_{i}}$ be the leading term of $\beta _{i}$, for the nonzero $\beta _{i}$.
We know that \[0= f(\sum_{i=1}^n\beta_i q_i). \]
 On the other hand, the leading term of the power series $ f(\sum_{i=1}^n\beta_i q_i)$ is a sum of some 
of the elements $\xi_{i}t^{j_i}e_{i}t^{k_{i}}$ and therefore is nonzero since $e_{1},\ldots ,e_{n}$ are linearly independent over $\mathbb C$ (by their construction). 
Therefore, all $\xi _{i}=0$, so $\beta _{i}=0$.

$ $

We will now show that all elements of $I$ are $\mathbb C\{t\}$-linear combinations
 of elements 
\[p_{k}-\sum_{i=1}^{n}(\zeta _{i,k}q_{i}+t\xi_{i,k}(t)q_{i})\]
for some $k\in \mathbb N, \zeta _{i,k}\in \mathbb C$, $\xi _{i,k}(t)\in {\mathbb C}\{t\}$ (recall that $p_{k}$ are monomials from $\mathbb C\langle x_{1}, \ldots , x_{\rho }\rangle $). 
 The proof  is similar to the proof  in \cite{DDS}, 
 with power series replacing polynomials. 
 Since $\operatorname {Im} (f)=\sum_{i=1}^{n}\mathbb C\{t\}f(q_{i})$, there are 
$\zeta _{i,k}\in \mathbb C$, $\xi _{i,k}(t)\in {\mathbb C}\{t\}$
 such that  \[p_{k}-\sum_{i=1}^{n}(\zeta _{i,k}q_{i}+t\xi_{i,k}(t)q_{i})\in I.\]
 We denote
\[z_{k}=\sum_{i=1}^{n}(\zeta _{i,k}q_{i}+t\xi_{i,k}(t)q_{i})   .\]
Hence \[p_{k}-z_{k}\in I.\]

 Let $r\in I.$ Then $r = \sum_{i=1}^{n'} s_{i}(t)p_i$ for some $s_i(t)\in \mathbb C \{t\}$ and some $n'$ (we consider ordinary ideals in ${\mathbb C}\langle x_{1}, \ldots , x_{\rho }\rangle \{t\}$ so $r$ is a finite sum of elements $s_{i}(t)p_{i}$).  
Observe that $p_i-z_i\in I$, so
\[\sum_{i=1}^{n'} s_i(t) z_i = r - \sum_{i=1}^{n'} s_i(t) (p_{i}-z_{i})\in I.\]

Note that 
\[\sum_{i=1}^{n'} s_i(t)z_i\in \sum_{i=1}^n q_i \mathbb C\{t\} ,\]
 hence \[\sum_{i=1}^{n'} s_i(t)z_i= \sum_{i=1}^n q_i\alpha_i\] 
for some $\alpha_i\in \mathbb C\{t\}$.  It follows that $\sum_{i=1}^n q_i\alpha_i\in I$. By the first part of this proof, $\alpha_i=0$.
Therefore,  $r - \sum_{i=1}^{n'} s_i(t) (p_{i}-z_{i})=0$, as required.
   
\end{proof}

\begin{prop}\label{prop:relnew} Let notation and assumptions be as in Subsection \ref{333}.
    Let $p_k$ be as in Notation \ref{not:monomials}. For any $k=1, 2, \ldots $, there exist $\zeta _{i,k}\in \mathbb C$, $\xi _{i,k}(t)\in {\mathbb C}\{t\}$ such that 
    \[p_{k}-\sum_{i=1}^{n}(\zeta _{i,k}+t\xi_{i,k}(t))q_{i}\in I.\]
    Moreover,  for any $k, m\in\{1,2,\dots, n\}$, there exist $\zeta _{i,k,m}\in \mathbb C$, $\xi _{i,k,m}(t)\in {\mathbb C}\{t\}$ such that 
\[q_{k}q_{m}-\sum_{i=1}^{n}(\zeta _{i,k,m}+t\xi_{i,k,m}(t))q_{i}\in I.\]
\end{prop}
\begin{proof} It follows from Proposition \ref{basesnew}. 
\end{proof}
\section{The algebra $\mathcal N$}\label{6}

  In this section we will define the algebra $\mathcal N$. We use the notation of Proposition \ref{prop:relnew}. We will use non-unital algebras and the First Isomorphism Theorem for non-unital algebras (see \cite{Bresar} for a reference on non-unital algebras).

\begin{prop}\label{8new}  Let $\mathbb C\langle y_1,\dots, y_n \rangle'\{t\}$ be the free non-unital $\mathbb C\{t\}$-algebra generated by free generators $y_{1}, \ldots , y_{n}$. Let $J'$ be the ideal generated in this algebra by 
 \[ y_{k}y_{m}-\sum_{i=1}^{n}(\zeta _{i,k,m}+t\xi_{i,k,m}(t))y_{i},\] for all $k,m\leq n$. 

   Let  $\mathcal N$ be the quotient algebra $\mathbb C\langle y_1,\dots, y_n \rangle' \{t\}/J'.$ 
 Then  $\mathcal N$ has an identity element $y_{j}+J'$ for some $j\leq n$.

 Moreover, $\mathcal N$ can be presented as a   
$\mathbb C\{t\}$-linear space  with  generators $d_{1},\dots, d_n$ which span ${\mathcal N}$ as a ${\mathbb C}\{t\}$-module  
 and with ${\mathbb C}\{t\}$-linear  associative multiplication given by  \[ d_{k}\ast _{t}d_{m}=\sum_{i=1}^{n}(\zeta _{i,k,m}+t\xi_{i,k,m}(t))d_{i},\]
  for all $k, m\leq n$. 

  As  a set ${\mathcal N}={\mathbb C}\{t\}d_{1}+\ldots +{\mathbb C}\{t\}d_{n}$. Moreover, the elements $d_{1}, \ldots , d_{n}$ are linearly independent over ${\mathbb C}\{t\}$. Let $D$ be a non-unital ${\mathbb C}\{t\}$-algebra generated by $d_{1}, \ldots , d_{n}$ with the above multiplication $\ast _{t}$.
 Then,  $d_{j}$ is the identity element in $D$. 
 
\end{prop}
\begin{proof} {\em Part $1$.} We will first show that the relations 
 \[d_{k}\ast _{t}d_{m}-\sum_{i=1}^{n}(\zeta _{i,k,m}+t\xi_{i,k,m}(t))d_{i}\] give a well-defined associative multiplication on
 ${\mathbb C}\{t\}d_{1}+ \cdots +{\mathbb C}\{t\}d_{n}$ (so they give a well-defined multiplication table).

 Recall from Proposition \ref{prop:relnew} that for any $k, m\in\{1,2,\dots, n\}$  $\zeta _{i,k,m}\in \mathbb C$, $\xi _{i,k,m}(t)\in {\mathbb C}\{t\}$ are such that 
\[q_{k}q_{m}-\sum_{i=1}^{n}(\zeta _{i,k,m}+t\xi_{i,k,m}(t))q_{i}\in I.\]
 Hence we obtain the following relations
\[f(q_{k})f(q_{m})-\sum_{i=1}^{n}(\zeta _{i,k,m}+t\xi_{i,k,m}(t))f(q_{i})=0.\]
  Fix $i',j',k'\in \{1, \ldots , n\}$. To see that  the multiplication table is well-defined (so that these relations give a  
well-defined associative multiplication) observe that by using 
 the formulas \[f(q_{k}) f(q_{m})-\sum_{i=1}^{n}(\zeta _{i,k,m}+t\xi_{i,k,m}(t))f(q_{i})\]
  for various $k,m$ several times (and without using any other formulas) we obtain 
\[(f(q_{i'})f(q_{j'}))f(q_{k'})=\sum_{l=1}^{n}c_{l}(t)f(q_{l}), f(q_{i'})(f(q_{j'})f(q_{k'}))=\sum_{l=1}^{n}c_{l}'(t)f(q_{l}),\] for some $c_{l}(t), c_{l}'(t)\in {\mathbb C}\{t\}$.

By associativity, we get
 $\sum_{l=1}^{n}(c_{l}(t)-c_{l}'(t))f(q_{l})=0$. By Proposition \ref{basesnew} (1)
we get $c_{l}(t)=c_{l}'(t)$.

 Because in the above  lines we only used  the formulas \[f(q_{k})f(q_{m})-\sum_{i=1}^{n}(\zeta _{i,k,m}+t\xi_{i,k,m}(t))f(q_{i})\]  to obtain the associativity, without using any other relations, it follows that the analogous results hold when we replace $f(q_{i})$ with $d_{i}$, since $d_{i}$ satisfy relations:
\[    d_{k}\ast _{t}d_{m}=\sum_{i=1}^{n}(\zeta _{i,k,m}+t\xi_{i,k,m}(t))d_{i}.\]
 Therefore, we obtain that the above formulas give a well-defined associative multiplication on 
 ${\mathbb C}\{t\}d_{1}+\ldots +{\mathbb C}\{t\}d_{n}$. 

Now this implies that $(d_{i'}\ast _{t}d_{j'})\ast _{t}d_{k'}=d_{i'}\ast _{t}(d_{j'}\ast _{t}d_{k'})$. This is because we can use the same steps when using relations  \[d_{k}\ast _{t} d_{m}-\sum_{i=1}^{n}(\zeta _{i,k,m}d_{i}+t\xi_{i,k,m}(t)d_{i})\] instead of relations \[f(q_{k})f( q_{m})-\sum_{i=1}^{n}(\zeta _{i,k,m}f(q_{i})+t\xi_{i,k,m}(t)f(q_{i})).\]

{\em Part $2$.}  We will first show that $D$ has an identity element $1_{D}$ which is $d_{j}$ for some $1\leq j\leq n$.   Recall that the identity element of a ring is unique. 
 By Proposition \ref{basisnew} (3), $f(q_{j})$ is the identity element in $Im(f)$. 
Therefore,  for each $i$, 
 \[f(q_{j})f(q_{i})=f(q_{i})=f(q_{i})f(q_{j}).\]
 Recall that \[f(q_{k})f(q_{m})=\sum_{i=1}^{n}(\zeta _{i,k,m}+t\xi_{i,k,m}(t))f(q_{i}),\] so by
 using these relations to simplify the above equation we obtain {\bf  Fact $1$}: the 
 elements  $f(q_{j})f(q_{i})-f(q_{i})$ and $f(q_{i})f(q_{j})-f(q_{i})$ are linear combinations of the elements 
\[f(q_{k})f(q_{m})-\sum_{i=1}^{n}(\zeta _{i,k,m}+t\xi_{i,k,m}(t))f(q_{i}).\]

 Indeed, we can substitute \[f(q_{k})f(q_{m}):=\sum_{i=1}^{n}(\zeta _{i,k,m}+t\xi_{i,k,m}(t))f(q_{i})\] in $f(q_{j})f(q_{i})-f(q_{i})$ at each place where it occurs. We obtain a linear combination
 over ${\mathbb C}\{t\}$ of the  elements  $f(q_{1}), \ldots , f(q_{n})$.  This linear combination must be zero. 
  This  follows because $f(q_{1}), \ldots , f(q_{n})$ are linearly independent over ${\mathbb C}\{t\}$ by Proposition 
\ref{basesnew}.

 Since the elements $d_{1}, \ldots , d_{n}$ satisfy relations 
$d_{k}\ast _{t} d_{m}=\sum_{i=1}^{n}(\zeta _{i,k,m}+t\xi_{i,k,m}(t))d_{i}$
 it follows (by Fact $1$) that they also satisfy relations:
$d_{j}\ast _{t}d_{i}-d_{i}$ and $d_{i}\ast _{t}d_{j}-d_{i}$.   
  Therefore, by the above,  the elements  $d_{j}\ast _{t}d_{i}-d_{i}$ and $d_{i}\ast _{t}d_{j}-d_{i}$ are linear combinations of the elements 
\[d_{k}\ast _{t}d_{m}-\sum_{i=1}^{n}(\zeta _{i,k,m}+t\xi_{i,k,m}(t))d_{i}.\]
 Therefore, 
 $d_{j}$ satisfies the axioms of the identity element, and so  is the identity element in $D$.

Observe also that since the relations
\[y_{k}y_{m}-\sum_{i=1}^{n}(\zeta _{i,k,m}+t\xi_{i,k,m}(t))y_{i}\]
 belong to $J'$. Similarly, it  follows that  $y_{j}+J'$ satisfies the axioms of the identity element in $\mathbb C\langle y_1,\dots, y_n \rangle' \{t\}/J'$, and so  is the unique identity element in $\mathbb C\langle y_1,\dots, y_n \rangle '\{t\}/J'$ (recall that by the assumptions of this proposition,  $\mathbb C\langle y_1,\dots, y_n \rangle' \{t\}$ is a non-unital algebra so there is no auxiliary unit $1$).

{\em Part $3$.} 
 We will now show  that the elements $d_{i}$ are linearly independent over ${\mathbb C}\{t\}
 $. 
 We will first show  that the elements $y_{i}+J'$ are linearly independent over ${\mathbb C}\{t\}
 $. 

  Consider a  homomorphism of $\mathbb C\{t\}$-algebras $g:{\mathbb C}\langle y_{1}, \ldots , y_{n}\rangle' \{t\}/J'\rightarrow \operatorname {Im} (f)$  given by $g(y_{i}+J')= f(q_{i})$. It is a well-defined homomorphism of unital algebras, since \[g(y_{k}y_{m}-\sum_{i=1}^{n}(\zeta _{i,k,m}+t\xi_{i,k,m}(t))y_{i}+J')=
f(q_{k})f(q_{m})-\sum_{i=1}^{n}(\zeta _{i,k,m}f(q_{i})+t\xi_{i,k,m}(t)f(q_{i}))=0.\] 

 Moreover,  the  image of the identity element 
$y_{j}+J'\in {\mathbb C}\langle y_{1}, \ldots , y_{n}\rangle '\{t\}/J'$ 
  is $g(y_{j}+J')=f(q_{j})$, and it is an identity element in  $\operatorname {Im} (f)$ by Part $2$. 

By Proposition \ref{basesnew} (1), the elements $f(q_{i})$ are linearly independent over ${\mathbb C}\{t\}$. Since $g(y_{i}+J')=f(q_{i})$, it follows that the elements $y_{i}+J'$ are linearly independent over ${\mathbb C}\{t\}$. Consequently, $d_{i}$ are linearly independent over ${\mathbb C}\{t\}$.

{\em Part $4$.} We will now show that the algebra $D$ is isomorphic to $\mathcal N$ as a unital algebra. 
  Consider a  homomorphism of $\mathbb C\{t\}$-algebras $ h:D\rightarrow {\mathbb C}\langle y_{1}, \ldots , y_{n}\rangle '\{t\}/J'$ given by $h(d_{i})=y_{i}+J'$. It is a well-defined homomorphism of non-unital algebras since 
$h(  d_{k}\ast _{t}d_{m}-\sum_{i=1}^{n}(\zeta _{i,k,m}+t\xi_{i,k,m}(t))d_{i}  )=y_{k}y_{m}-\sum_{i=1}^{n}(\zeta _{i,k,m}+t\xi_{i,k,m}(t))y_{i} +J'= 0+J'$. 
  By Part $3$, the  elements $y_{i}+J'$ are linearly independent over ${\mathbb C}\{t\}$. It follows that $d_{i}$ are linearly independent over ${\mathbb C}\{t\}$. Therefore $D$ is a free ${\mathbb C}\{t\}$-module of rank $n$. Moreover $h$ is injective, since if $h(\sum_{i=1}^{n}\sigma _{i}(t)d_{i})=0$, where $\sigma _{i}(t)\in {\mathbb C}\{t\}$, then $\sum_{i=1}^{n}\sigma _{i}(t)y_{i}+ J'=0+J'$, and since  
 $y_{i}+J'$ are linearly independent over ${\mathbb C}\{t\}$, all $\sigma _{i}(t)=0$.
 
The relations $y_{k}y_{m}-\sum_{i=1}^{n}(\zeta _{i,k,m}+t\xi_{i,k,m}(t))y_{i} $ 
reduce every word of length
at least two to a ${\mathbb C}\{t\}$-linear
combination of $y_{1}, \ldots  , y_{n}$;
hence the quotient is
spanned by $y_{i} + J'$, all of 
which lie in $\operatorname {Im} (h)$. Since $y_{j}+J'$ is an identity element in ${\mathbb C}\langle y_{1}, \ldots , y_{n}\rangle '\{t\}/J'$, and the identity element is unique, and since $y_{i}+J'\in  \operatorname {Im} (h)$, it follows that the identity element of 
 $ {\mathbb C}\langle y_{1}, \ldots , y_{n}\rangle '\{t\}/J'$ is in $\operatorname {Im} (h)$, so $h$ is onto.  
 
Since $D$ has an identity element $1_{D}$ and $h$ is onto, $h(1_{D})$ satisfies the axioms of the identity element in $\operatorname {Im} (h)$. Therefore $h$ is an isomorphism of unital algebras.
  By the First Isomorphism Theorem for algebras,  
 ${\mathbb C}\langle y_{1}, \ldots , y_{n}\rangle '\{t\}/J'$ is isomorphic to $D$. Therefore, $\mathcal N$ is isomorphic to $D$, as required. 
\end{proof}

\begin{proposition}\label{Im} The ${\mathbb C}\{t\}$-algebra $\mathcal N$ is isomorphic to $\operatorname {Im} (f)$. 
\end{proposition}
\begin{proof} We will use the presentation of $\mathcal N$ as a ${\mathbb C}\{t\}$-algebra generated by $d_{1}, \ldots , d_{n}$ subject to relations 
 $d_{k} \ast_{t}d_{m}=\sum_{i=1}^{n}(\zeta _{i,k,m}d_{i}+t\xi_{i,k,m}(t)d_{i})$.

Observe that by Proposition \ref{8new} the element $d_{j}$ is the unique identity element in $\mathcal N$ (since the identity element in any ring is unique) so every element in $\mathcal N$ is of the form $\sum_{i=1}^{n}d_{i}\sigma _{i}(t)$ for some $\sigma _{i}(t)\in {\mathbb C}\{t\}$, including the identity element.

 Consider a homomorphism of ${\mathbb C}\{t\}$-algebras 
$h:{\mathcal N}\rightarrow \operatorname {Im} (f)$ given by $h(d_{i})=f(q_{i})$. This homomorphism is well-defined since \[h(d_{k} \ast_{t}d_{m}-\sum_{i=1}^{n}(\zeta _{i,k,m}+t\xi_{i,k,m}(t))d_{i})=f(q_{k}) f(q_{m})-\sum_{i=1}^{n}(\zeta _{i,k,m}+t\xi_{i,k,m}(t))f(q_{i})=0.\]

 Observe that $\ker (h)=0$, since if $h(\sum_{i=1}^{n}\sigma _{i}(t)d_{i})=0$ for some $\sigma _{i}(t)\in {\mathbb C}\{t\}$ then    $0= h(\sum_{i=1}^{n}\sigma _{i}(t)d_{i})= \sum_{i=1}^{n}\sigma _{i}(t)f(q_{i})=0$. By Proposition \ref{basesnew} (1) this implies that all $\sigma _{i}(t)=0$.

Observe that $h$ is onto. Since ${\mathbb C}\{t\}f(q_{1})+\cdots +{\mathbb C}\{t\}f(q_{n}) \subseteq \operatorname {Im} (h)$ because $h(d_{i})=f(q_{i})$, the ${\mathbb C}\{t\}$-span $\sum_{i=1}^{n} f(q_{i}){\mathbb C}\{t\}$ is contained in  $\operatorname {Im} (h)$;  by Proposition  \ref{basisnew} this span equals  $ \operatorname {Im} (f)$. Therefore, $ \operatorname {Im} (f)\subseteq  \operatorname {Im} (h)$.  

 By Proposition \ref{8new}, $d_{j}$
 is the identity element in $\mathcal N$. As shown in the proof of 
Proposition \ref{basisnew} (3), $f(q_{j})=h(d_{j})$ is the identity element in $\operatorname {Im} (f)$, so $h$ is a surjective homomorphism of unital algebras.
By the First Isomorphism Theorem for ${\mathbb C}\{t\}$-algebras we obtain, since $\ker h$ is trivial, 
\[{\mathcal N}\cong {\mathcal N}/\ker h\cong   \operatorname {Im} (h).\] 
\end{proof}
 \section{The algebra $N$}\label{N*}

 In this section we will  define the algebra $N$. The construction is similar to that in \cite{DDS}. We use the  notation of  Proposition \ref{prop:relnew}.  
\begin{notation}\label{bar q} Denote \[q_{i}=\bar {q}_{i}+{\tilde {q}}_{i},\]
 where $\bar {q}_{i}\in \mathbb C\langle x_{1}, \ldots , x_{\rho } \rangle$ and $\tilde {q}_{i}\in t\mathbb C\langle x_{1}, \ldots , x_{\rho } \rangle \{t\}$.
 Notice that 
\[ p_{k}-\sum_{i=1}^{n}\zeta _{i,k}\bar {q}_{i}\in I+t\mathbb C\langle x_{1}, \ldots , x_{\rho } \rangle \{t\}.\]
\end{notation}

\begin{notation}\label{not:N} Let notation and assumptions be as in Subsection \ref{333}.
We denote by $J$ the set consisting of $\mathbb C$-linear combinations of elements 
\begin{equation}\label{eq:J}
    p_{k}-\sum_{i=1}^{n}\zeta _{i,k}\bar {q}_{i}\in \mathbb C\langle x_{1}, \ldots , x_{\rho } \rangle,
\end{equation}
where $\zeta _{i,k}$ are as in Proposition \ref{prop:relnew}.
 Let $\langle J\rangle $ be the ideal of $\mathbb C\langle x_{1}, \ldots , x_{\rho } \rangle$ generated by elements from $J$. Let $N$ be the quotient algebra $ \mathbb C\langle x_{1}, \ldots , x_{\rho } \rangle/\langle J\rangle$. 
 Notice that by Proposition \ref{prop:relnew}
\[{\bar {q}_{k}}{\bar {q}_{m}}-\sum_{i=1}^{n}\zeta _{i,k,m}{\bar {q}_{i}}\in  (I+t\mathbb C\langle x_{1}, \ldots , x_{\rho }\rangle \{t\}) \cap {\mathbb C}\langle x_{1}, \ldots , x_{\rho }\rangle.\]
 \end{notation}

\begin{lemma}\label{2new} Let $e\in \mathbb C\langle x_{1}, \ldots , x_{\rho  } \rangle$.
 We have $e\in \langle J \rangle$ if and only if  $e+e'\in I$ for some $e'\in t\mathbb C\langle x_{1}, \ldots , x_{\rho  } \rangle \{t\}$. In particular, $\langle J\rangle \subseteq I+t\mathbb C\langle x_{1}, \ldots , x_{\rho } \rangle \{t\}$.  
\end{lemma} 
\begin{proof} The proof is the same  as that  of Lemma $3.4$ in \cite{DDS} (the only difference is that we use power series instead of polynomials). We repeat it for convenience of the reader: 
 If $e\in J$, then by Notation \ref{not:N}, $e = \sum_k \beta_k (p_k-l_k)$ where $l_k = \sum_{i=1}^{n}\zeta _{i,k}\bar {q}_{i} $ and $\beta_k\in \mathbb C$. Then let 
    \[{l'_k} = \sum_{i=1}^{n} \zeta_{i, k}{\tilde q_i}+\sum_{i=1}^{n} t\xi_{i, k}(t)q_i. \]
     Then $e'= \sum_i \beta_i {l'_i}$ is such that $e+e'\in I$ by Proposition \ref{prop:relnew}. If $e$ is a finite sum of elements from  ${\mathbb C}\langle x_{1}, \ldots , x_{\rho }\rangle \cdot J\cdot {\mathbb C}\langle x_{1}, \ldots , x_{\rho }\rangle $ then $e+e'\in I$ for an  appropriate $e'$ by the above.

    Suppose, on the other hand, that $e\in \mathbb C\langle x_{1}, \ldots , x_{\rho }\rangle$ is such that $e+e'\in I$ for some $e'\in t\mathbb C\langle x_{1}, \ldots , x_{\rho } \rangle \{t\}$. Then by Proposition \ref{basesnew} (3) we get that
    \[e+e' \in \sum_k\mathbb C\{t\}(p_k-t_k),\]
    where $t_k=\sum_{i=1}^{n} (\zeta_{i,k} + t\xi_{i,k}(t))q_i$. So 
    $e+e' = \sum_k \alpha_k (p_k-t_k)$ 
    for some $\alpha_k\in \mathbb C\{t\}$. Let $\alpha_k = r_k+{\bar r}_k$ where $r_k\in \mathbb C$ and ${\bar r}_k\in t\mathbb C\{t\}$. Denote $l_k = \sum_{i=1}^n \zeta_{i, k}{\bar q}_i$. By  
 comparing constant terms it follows that 
    \[e = \sum_k r_k(p_k-l_k).\]
   Recall that  $J$ denotes the
$\mathbb C$-linear span of the
displayed elements. So $e\in J\subseteq \langle J\rangle$, as required.

    Notice that if $e\in \langle J\rangle $ then $e+e'\in I$, for some $e'\in t\mathbb C\langle x_{1}, \ldots , x_{\rho } \rangle \{t\}$ by a previous result of this proof. Therefore, for some $i\in I$,  
    \[e=i-e' \in I+t\mathbb C \langle x_{1}, \ldots , x_{\rho } \rangle \{t\}.\]
    Note that $e$ is an arbitrary element of $\langle J\rangle $, hence 
    $\langle J\rangle\subseteq  I+t\mathbb C\langle x_{1}, \ldots , x_{\rho } \rangle \{t\}.$

\end{proof}

\begin{prop}\label{3newest} The dimension of $N$ equals $n$. Moreover, the elements $\bar {q}_{k}+\langle J\rangle\in N$ are a basis of $N$ as a $\mathbb C$-vector space.
\end{prop}

\begin{proof}  We use the same proof as in Proposition $3.5$ in \cite{DDS}, but we  use power series instead of polynomials. Recall that by Proposition \ref{basisnew}
 \[\mathbb C\langle x_{1}, \ldots , x_{\rho} \rangle\subseteq \sum_{k=1}^{n}\mathbb C \{t\}{ {q}_{k}}+ I.\]

Recall that $q_{k}-{\bar q}_{k}\in t\mathbb C\langle x_{1}, \ldots , x_{\rho } \rangle\{t\}$. 
Moreover, ${\bar q}_{k}\in \mathbb C\langle x_{1}, \ldots , x_{\rho } \rangle$.  Notice also that $ J \subseteq \mathbb C\langle x_{1}, \ldots , x_{\rho }\rangle $. If $u\in I$, write  
  $u=u_{0}+(u-u_{0})$, where $u-u_{0}$ is divisible by $t$.  Lemma \ref{2new} applied to $u_{0}$ shows $u_{0}\in \langle J\rangle.$ 
It follows that
\[\mathbb C\langle x_{1}, \ldots , x_{\rho } \rangle\subseteq
 \sum_{k=1}^{n}\mathbb C {\bar {q}_{k}}+\langle J\rangle+t\mathbb C\langle x_{1}, \ldots , x_{\rho } \rangle\{t\},\] 
 hence by comparing terms at $t^{0}$ we have 
 $\mathbb C\langle x_{1}, \ldots , x_{\rho }\rangle\subseteq
 \sum_{k=1}^{n}\mathbb C {\bar {q}_{k}}+\langle J\rangle .$
 Therefore, 
$\{{\bar {q}_{k}}+\langle J\rangle\mid k\in \{1,2, \ldots ,n\}\}$  span $N$ as a $\mathbb C$-vector space.
 Therefore  the dimension of $N$ does not exceed $n$.

We now show that the elements ${\bar {q}_{k}}+\langle J\rangle\in N$ for $k\in \{1,2, \ldots ,n\}$ are linearly independent over $\mathbb C$.
Suppose on the contrary that we have
\[\sum_{i=1}^{n}\xi_{i}{\bar {q}_{i}}\in \langle J\rangle \]
for some $\xi _{i}\in \mathbb {C}$. By Lemma \ref{2new} and Proposition \ref{basisnew} we have  \[\langle J\rangle \subseteq I+t\mathbb C\langle x_{1}, \ldots , x_{\rho }\rangle \{t\}\subseteq I+t\sum_{i=1}^{n}{\mathbb C}\{t\}{ {q}_{i}}.\]
 Recall that \[q_{i}-\bar {q}_{i}\in t\mathbb C\langle x_{1}, \ldots , x_{\rho } \rangle \{t\}\subseteq t(\sum_{i=1}^{n}\mathbb C \{t\}q_{i}+ I).\]
 Therefore $\sum_{i=1}^{n}\xi_{i}{\bar {q}_{i}}\in \langle J\rangle $ implies
\[\sum_{i=1}^{n}\xi_{i}{ {q}_{i}}\in \langle J\rangle  + t(\sum_{i=1}^{n}\mathbb C\{t\}q_{i}+ I)\subseteq I+t\sum_{i=1}^{n}{\mathbb C}\{t\}{ {q}_{i}}. \]
 It follows that there is $e'\in t\sum_{i=1}^{n}{\mathbb C}\{t\}q_{i}$ such that $\sum_{i=1}^{n}\xi _{i}q_{i}-e'\in I$. If $e'=t\sum_{i=1}^{n}\sigma _{i}(t)q_{i}$ then $\sum_{i=1}^{n}(\xi _{i}-t\cdot \sigma _{i}(t))q_{i}\in I$. Proposition \ref{basesnew} (1) gives $\xi _{i}-t\cdot \sigma _{i}(t)=0$ for all $i$. Since $\xi _{i}\in \mathbb C$ and  $t\cdot \sigma _{i}(t)\in t\cdot \mathbb C\{t\}$, all $\xi _{i}=0$, contradicting the assumed non-trivial linear relation. 
\end{proof}
\begin{prop}\label{7new} Let $\mathbb C\langle y_1,\dots, y_n \rangle'$ be the free non-unital  $\mathbb C$-algebra generated by free generators $y_{1}, \ldots , y_{n}$. Let $J''$ be the ideal generated in this algebra by the elements  \[ y_{k}y_{m}-\sum_{i=1}^{n}\zeta _{i,k,m}y_{i},\] for all $k, m\leq n$. Let  $R$ be the quotient $\mathbb C$-algebra $\mathbb C\langle y_1,\dots, y_n \rangle'/ J''$. Then  $y_{j}+J''$ is the identity element and $N$ is isomorphic (as a unital $\mathbb C$-algebra) to  $R$ (where ${j}$ is as in Proposition \ref{8new}).  
 
 Moreover, $N$ can be presented as  a 
$\mathbb C$-vector space  with  generators $d_{1},\dots, d_n$ which span $N$ as a $\mathbb C$-vector space 
 and with $\mathbb C$-linear  associative multiplication given by  \[ d_{k}\cdot d_{m}=\sum_{i=1}^{n}\zeta _{i,k,m}d_{i},\]
  for all $k, m\leq n$. Denote by $D'$ a  non-unital algebra on the set 
${\mathbb C}d_{1}+\ldots +{\mathbb C}d_{n}$ with multiplication $\cdot $.
 Then, $D'$ has  an identity element  $d_{j}$, where ${j}$ is as in Proposition \ref{8new}.  As a set $N={\mathbb C}d_{1}+\ldots +{\mathbb C}d_{n}$, and the  elements $d_{1}, \ldots , d_{n}$ are linearly independent over $\mathbb C$.
\end{prop}
\begin{proof}  We use the same  proof as in Proposition $3.6$ of  \cite{DDS}, but we adjust it to the context of  power series, and add additional explanations about the identity element $1_{N}$. 

 Let \[\zeta:\mathbb C\langle y_1,\dots, y_n \rangle'\rightarrow N\]
  be a homomorphism of non-unital $\mathbb C$-algebras defined as $\zeta (y_{i})={\bar {q}_{i}}+\langle J\rangle$.
By Proposition \ref{3newest},  elements $\bar {q}_{k}+\langle J\rangle\in N$ are a basis of $N
$ as a $\mathbb C$-vector space; therefore $\zeta $ is onto.

 Let \[{\tilde J} =\ker (\zeta )\subseteq \mathbb C\langle y_1,\dots, y_n \rangle'.\]
 Then by the First Isomorphism Theorem for non-unital  $\mathbb C$-algebras  $\mathbb C\langle y_1,\dots, y_n \rangle'/{\tilde J}$ is isomorphic (as a non-unital ${\mathbb C}$-algebra) to 
$ Im (\zeta )=N$. Observe that $J''\subseteq {\tilde J}$. To show that,  
 recall that $J''$ is the ideal of $\mathbb C\langle y_1,\dots, y_n \rangle '$ generated by  elements 
\[ y_{k}y_{m}-\sum_{i=1}^{n}\zeta _{i,k,m}y_{i}\]
where $\zeta_{i,k,m}$ are as in Proposition \ref{prop:relnew}. Observe that 
\[\zeta \left (y_{k}y_{m}-\sum_{i=1}^{n}\zeta _{i,k,m}y_{i} \right)={\bar {q}_{k}}{\bar {q}_{m}}-\sum_{i=1}^{n}\zeta _{i,k,m}{\bar {q}_{i}}+\langle J\rangle =0+\langle J\rangle .\] 
 Therefore $J''\subseteq {\tilde J}$.

  We next show that $\mathbb C\langle y_{1}, \ldots , y_{n } \rangle'/J''$ has an identity element which is of the form $\sum_{i=1}^{n}{\alpha }_{i}y_{i}+J''$. By Proposition \ref{8new} the algebra $\mathcal N$ has an identity element $1_{\mathcal N}= { y}_{j} + J'$. Recall from Proposition \ref{8new} that ${\mathcal N}=\mathbb C\langle y_{1}, \ldots , y_{n }\rangle' /J'$, where $\mathbb C\langle y_{1}, \ldots , y_{n }\rangle' $ is a non-unital algebra (as in our current proposition). 
  Since $1_{\mathcal N}$ satisfies the axioms of the identity element, it follows that the elements ${ y}_{i}y_{j}-y_{i}$ and $y_{j}y_{i}-y_{i}$ belong to $J'$. Recall that $J'$ is generated by relations  $y_{k}y_{m}-\sum_{i=1}^{n}(\zeta _{i,k,m}+t\xi_{i,k,m}(t))y_{i},$ for all $k,m\leq n$. 
 By comparing elements at $t^{0}$, we obtain that 
   ${ y}_{i}y_{j}-y_{i}$ and $y_{j}{ y}_{i}-y_{i}$ belong to the ideal generated by  $y_{k}y_{m}-\sum_{i=1}^{n}\zeta _{i,k,m}y_{i}$ in ${\mathbb C}\langle y_{1}, \ldots , y_{n}\rangle'$, so  they belong to $J''$. It follows that 
 ${ y}_{j}+J''$ is the identity element in $\mathbb C\langle y_{1}, \ldots , y_{n } \rangle'/J''$ (since the identity element is unique in any ring). 
Since $J''\subseteq {\tilde J}$, it follows that 
 ${y}_{j}+ {\tilde J}$ is the identity element in $\mathbb C\langle y_{1}, \ldots , y_{n } \rangle'/{\tilde J}$ (since it satisfies the axioms of the identity element).
 
 Therefore, $c=\zeta ({ y}_{j})$ satisfies the axioms of the identity element in $\operatorname {Im} (\zeta )$.  We have shown that $\zeta $ is onto $N$, and so \[c={\bar q}_{j}+\langle J\rangle \] is the identity element in ${\mathbb C}\langle x_{1}, \ldots , x_{\rho }\rangle/\langle J\rangle =N$. 
Recall that $\zeta $ is a homomorphism of non-unital algebras, so 
$\mathbb C\langle y_1,\dots, y_n \rangle'/{\tilde J}$ is isomorphic as a non-unital ${\mathbb C}$-algebra to 
$ Im (\zeta )=N$. Since an isomorphism of non-unital algebras preserves the unit  axioms when a unit exists, 
$\mathbb C\langle y_1,\dots, y_n \rangle'/{\tilde J}$ is also  isomorphic as a unital ${\mathbb C}$-algebra to 
$ Im (\zeta )=N$. 

Recall that  $J''\subseteq {\tilde J}$. It follows that  the dimension of $ \mathbb C\langle y_1,\dots, y_n \rangle'/{\tilde J}$ does not exceed the dimension of $ \mathbb C\langle y_1,\dots, y_n \rangle'/J''$. 
 
 We now show that ${\tilde J}=J''$. Notice that $\mathbb C\langle y_1,\dots, y_n \rangle'/J''$ is at most $n$-dimensional 
since every element can be presented as a linear combination of elements $y_{i}+J''$ for $i=1,2, \ldots, n$ (since we assumed that  $\mathbb C\langle y_1,\dots, y_n \rangle'$ is a non-unital algebra). On the other hand 
$\mathbb C\langle y_1,\dots, y_n \rangle'/{\tilde J}$ is isomorphic to 
 $\mathbb C\langle x_{1}, \ldots , x_{\rho} \rangle/\langle J \rangle $, and hence is $n$-dimensional (by Proposition \ref{3newest}). Hence, by comparing dimensions, we get  that ${\tilde J}=J''$, as required. This concludes the proof that $N$ is isomorphic to $R$. 
 
To obtain the  presentation of $N$ using elements $d_{i}$ (subject to relations $d_{k}d_{m}=\sum_{i=1}^{n}\zeta _{i,k,m}d_{i}$)
observe that if we denote $y_{i}+J'':=d_{i}$ in $R$ then we obtain such a presentation of $N$  using generators and relations. 
Then $d_{1}, \ldots , d_{n}$ generate a $\mathbb C$-algebra isomorphic to $N$. Moreover, $d_{j}$ is the identity element in this algebra, since it satisfies the axioms of the identity element,  
   because $y_{j}+J''$ is the identity element in $\mathbb C\langle y_{1}, \ldots , y_{n }\rangle' / J''$. 
\end{proof}

\subsection{How to construct  $N$ satisfying prescribed relations}\label{pr} 
 In this section we show that $N$ is isomorphic to $Im(f)/t\cdot Im(f)$. 
 
\begin{lemma}\label{condition1} 
 Let notation be as in Proposition \ref{7new}. 
 Let $J'=\{s\in {\mathbb C}\langle x_{1}, \dots , x_{\rho }\rangle : f(s)\in t\cdot {\operatorname {Im} (f)}\}$. 
 Then  ${\mathbb C}\langle x_{1}, \ldots , x_{\rho }\rangle/ J'$ is isomorphic to $N$. 
\end{lemma}
\begin{proof} Observe first that $J'$ is a two-sided ideal in ${\mathbb C}\langle x_{1}, \ldots , x_{\rho }\rangle$ since 
 if $s\in J'$ and $r,r'\in {\mathbb C}\langle x_{1}, \ldots , x_{\rho }\rangle$ then $f(rsr')=f(r)f(s)f(r')\in t\cdot {\operatorname {Im} (f)}$, so $rsr'\in J'$. Closure of $J'$ under sums and scalar multiplication is immediate.  

Let $d_{1}, \ldots , d_{n}$ be as in Proposition \ref{7new}, and denote the multiplication there by  $\ast _{0}$. By Proposition \ref{7new}, $N$ can be presented as the $\mathbb C$-algebra  generated by  $d_{1}, \ldots , d_{n}$ with multiplication  $\ast _{0}$
 given by $d_{k}\ast_{0}d_{m}=\sum_{i=1}^{n}\zeta _{i,k,m}d_{i},$
 where $\zeta _{i,k,m}$ are such that 
$f(q_{k})f(q_{m})=\sum_{i=1}^{n}(\zeta _{i,k,m}+t\xi _{i,k,m}(t))f(q_{i})$.

 Therefore, it suffices  to show that ${\mathbb C}\langle x_{1}, \ldots , x_{\rho }\rangle/ J'$ is isomorphic to this algebra.

Consider the  map  $\beta:  {\mathbb C}\langle x_{1}, \ldots , x_{\rho }\rangle \rightarrow N$ such that for $p\in   {\mathbb C}\langle x_{1}, \ldots , x_{\rho }\rangle $, \[\beta (p)=\sum_{i=1}^{n}{\alpha }_{i}d_{i}\]  if and only if 
 \[p-\sum_{i=1}^{n}{\alpha }_{i}{ q}_{i}\in I+t{\mathbb C}\langle x_{1}, \ldots , x_{\rho }\rangle \{t\} .\] 
 This is equivalent to the property that 
 \[f(p-\sum_{i=1}^{n}{\alpha }_{i}{ q}_{i})\in t\cdot \operatorname {Im} (f).\]
Observe that  such $\alpha _{i}\in \mathbb C$ exist by Proposition \ref{basisnew}.
 Indeed, write $p-\sum_{i=1}^{n}a_{i}(t)q_{i}\in I$ with $a_{i}(t)\in {\mathbb C}\{t\}$; put $\alpha _{i}=a_{i}(0)$. This gives $p-\sum_{i=1}^{n}\alpha _{i}q_{i}\in I+t\cdot {\mathbb C}\langle x_{1}, \ldots , x_{\rho }\rangle \{t\}$.

 Observe that such $\alpha _{i}\in \mathbb C$ are unique, since if $f(p-\sum_{i=1}^{n}{\alpha }_{i}{ q}_{i})\in t\cdot \operatorname {Im} (f)$ and $f(p-\sum_{i=1}^{n}{\alpha }_{i}'{ q}_{i})\in t\cdot \operatorname {Im} (f)$ then $\sum_{i=1}^{n}(\alpha _{i}-\alpha _{i}')f(q_{i})\in t\cdot \operatorname {Im} (f)$. Recall that $\operatorname {Im} (f)=\sum_{i=1}^{n}{\mathbb C}\{t\}f(q_{i})$. This implies that $\alpha _{i}-\alpha _{i}'\in {t\mathbb C}\{t\}$ (since by Proposition \ref{basesnew}  $f(q_{1}), \ldots , f(q_{n})$ are linearly independent over ${\mathbb C}\{t\}$). 
 Since $\alpha _{i}- \alpha _{i}'\in {\mathbb C}\cap t{\mathbb C}\{t\}$, we have $\alpha _{i}=\alpha _{i}'$.

 Observe that this map is onto, as $\beta ({\bar q}_{i})=d_{i}$.
  We will now describe the kernel of $\beta $. Let $c$ be such that $\beta (c) = 0$. Recall that $\beta (c)=\sum_{i}\alpha _{i}d_{i}$. Since the $d_{i}$ are linearly independent, $\beta (c)=0$ if and only if all $\alpha _{i}=0$, equivalently $f(c)\in  t\cdot \operatorname {Im} (f)$.

 We will now show that $\beta $ is a homomorphism of $\mathbb C$-algebras. 
 Observe that for $p, r\in  {\mathbb C}\langle x_{1}, \ldots , x_{\rho }\rangle $ if
 \[f(p)=\sum_{i=1}^{n}{\alpha }_{i}f(q_{i})+i_{1}, f(r)=\sum_{i=1}^{n}{\alpha }_{i}'f(q_{i})+ i_{2}\] for some $i_{1}, i_{2}\in t\cdot \operatorname {Im} (f)$ then
 \[\beta (p)=\sum_{i=1}^{n}{\alpha }_{i}d_{i}, \beta(r)=\sum_{i=1}^{n}{\alpha }_{i}'d_{i},\]
 hence, 
 $\beta (p+r)=\beta (p)+\beta (r)$, $\beta (p-r)=\beta (p)-\beta (r)$, $\beta (0)=0$.
 Observe also that  
\[f(pr)=f(\sum_{i=1}^{n}\sum_{j=1}^{n}\alpha _{i}\alpha _{j}'q_{i}q_{j})+i_{3}\] for some $i_{3}\in  t\cdot \operatorname {Im} (f).$  
 By Proposition \ref{prop:relnew}, 
\[f(pr)=\sum_{i=1}^{n}\sum_{j=1}^{n}\sum_{l=1}^{n}\alpha _{i}\alpha _{j}'\zeta _{l,i,j}f(q_{l})+i_{3}.\]
 
Therefore \[\beta (pr)=\sum_{i=1}^{n}\sum_{j=1}^{n}\sum_{l=1}^{n}\alpha _{i}\alpha _{j}'\zeta _{l,i,j}d_{l}.\]

 Observe, on the other hand, that
\[\beta (p)\ast _{0}\beta (r)=\sum_{i=1}^{n}\sum_{j=1}^{n}\alpha _{i}\alpha _{j}'d_{i}\ast _{0}d_{j},\] by the formula \[d_{k}\ast_{0}d_{m}=\sum_{i=1}^{n}\zeta _{i,k,m}d_{i}.\]  
 By the First Isomorphism Theorem  for $\mathbb C$-algebras, $\operatorname {Im} (\beta )$ is isomorphic to the algebra 
${\mathbb C}\langle x_{1}, \ldots , x_{\rho }\rangle /\ker(\beta ).$ We know that $\operatorname {Im} (\beta )$ is isomorphic to 
 $N$ (because it contains $d_{1}, \ldots , d_{n}$ and is a subalgebra of the algebra generated by $d_{1}, \ldots , d_{n}$ which is isomorphic to $N$). 
It remains to show that $c\in \ker (\beta )$ if and only if $f(c)\in t\cdot \operatorname {Im} (f)$ (so $\ker (\beta)=J'$).

Note that $\beta (c)=0$ if and only if  $f(c)\in t\cdot \operatorname {Im} (f)$. 
This concludes the proof. 
\end{proof}

\begin{lemma}\label{condition2}
   Let $m$ be a natural number. Let $s_{1}, \ldots, s_{m}\in {\mathbb C}\langle x_{1}, \ldots , x_{\rho }\rangle $ be such that $f(s_{i})\in t\cdot \operatorname {Im} (f)$.
 Let $S$ be the ideal generated by $s_{1}, \ldots , s_{m}$ in  ${\mathbb C}\langle x_{1}, \ldots , x_{\rho }\rangle $.   
 Suppose that the dimension of the algebra 
 $ {\mathbb C}\langle x_{1}, \ldots , x_{\rho }\rangle /S$ is at most $n$. Then 
 $ {\mathbb C}\langle x_{1}, \ldots , x_{\rho }\rangle /S$ is isomorphic to $N$. 
\end{lemma}
\begin{proof}
 Recall that $N$ has dimension $n$. 
 Let $J'$ be as in Lemma \ref{condition1}. Then $s_{1}, \ldots , s_{m}\in J'$ since $f(s_{i})\in t\cdot \operatorname {Im} (f)$. So $S\subseteq J'$. Since $S \subseteq  J'$, 
there is a surjection
${\mathbb C}\langle x_{1}, \ldots , x_{\rho }\rangle /S \rightarrow  {\mathbb C}\langle x_{1}, \ldots , x_{\rho }\rangle /J'$. 
 The target has dimension $n$, while the source has
dimension at most $n$; hence
both have dimension $n$, and
the map is an isomorphism.  
 By Lemma \ref{condition1} $ {\mathbb C}\langle x_{1}, \ldots , x_{\rho }\rangle /S$
 is isomorphic to $N$.  
\end{proof}

\begin{proposition}\label{imp} Let notation be as in
 Lemma \ref{condition1}. Then $N$ is isomorphic as a  $\mathbb C$-algebra to the quotient algebra $ \operatorname {Im} (f)/t\cdot \operatorname {Im} (f)$.
\end{proposition}
\begin{proof} It suffices to prove that 
${\mathbb C}\langle x_{1}, \ldots , x_{\rho }\rangle/ J'$ is isomorphic to $\operatorname {Im} (f)/t\cdot \operatorname {Im} (f)$. 
  Consider the homomorphism $g: {\mathbb C}\langle x_{1}, \ldots , x_{\rho }\rangle\rightarrow \operatorname {Im} (f)/t\cdot \operatorname {Im} (f)$ of $\mathbb C$-algebras given by $g(r)=f(r)+t\cdot \operatorname {Im} (f)$ for $r\in {\mathbb C}\langle x_{1}, \ldots , x_{\rho }\rangle$. 
By the First Isomorphism Theorem for algebras $\operatorname {Im} (g)$ is isomorphic to ${\mathbb C}\langle x_{1}, \ldots , x_{\rho }\rangle/\ker(g)$. We will now show that $ker(g)=J'$. 
 Let $r\in J'$. Then $f(r)\in t \cdot \operatorname {Im} (f)$, so $r\in \ker (g)$. 
  Suppose that $r\in \ker (g)$. Then $f(r)+t\cdot \operatorname {Im} (f)=g(r)=0$ so $f(r)\in t\cdot \operatorname {Im} (f)$; hence $r\in J'$. It remains to show that $g$ is onto. By Proposition \ref{basisnew}, $\operatorname {Im} (f)= \sum_{i=1}^{n}f(q_{i}){\mathbb C}\{t\}$, so $\operatorname {Im} (f)/t\cdot \operatorname {Im} (f)= \sum_{i=1}^{n}{\mathbb C}(f(q_{i})+t\cdot \operatorname {Im} (f))$. Note that $f(q_{i})+t\cdot \operatorname {Im} (f)=f({\bar q}_{i})+f({\tilde q}_{i})+ t\cdot \operatorname {Im} (f)=f({\bar q}_{i})+ t\cdot \operatorname {Im} (f)=g({\bar q}_{i})\in  \operatorname {Im} (g)$. Therefore, $g$ is onto. 
\end{proof}
\section{Defining  a formal deformation of $N$}\label{Z}

We now define a multiplication on $N$ which gives a formal deformation of $N$. 

\begin{theorem} \label{thm formal deformationnew}
Let $d_1,\dots, d_n$ be the basis elements from Proposition \ref{7new} and suppose
 $\zeta _{i,k,m}\in \mathbb C$, $\xi _{i,k,m}(t)\in {\mathbb C}\{t\}$ are as in Proposition \ref{prop:relnew}. Then the multiplication rule
\[d_{k} \ast_t d_{m}=\sum_{i=1}^{n}(\zeta _{i,k,m}d_{i}+t\xi_{i,k,m}(t)d_{i})\]
 gives a formal deformation such that $\ast_0$ gives the multiplication on $N$. 
 Moreover, the image of the identity element in $N$ is the identity element in the obtained deformation.
\end{theorem}
\begin{proof} Recall that $N$ is the $\mathbb C$-linear space with basis  $d_{1}, \ldots , d_{n}$ and associative $\mathbb C$-linear multiplication 
\[d_{k} d_{m}=\sum_{i=1}^{n}\zeta _{i,k,m}d_{i}.\] Moreover $N=d_{1}{\mathbb C}+\ldots +d_{n}{\mathbb C}$. 
 Define the multiplication $\ast _{t}$  on the set $d_{1}{\mathbb C}\{t\}+\ldots +d_{n}{\mathbb C}\{t\}$, as follows: 
\[d_{k} \ast_t d_{m}=\sum_{i=1}^{n}(\zeta _{i,k,m}+t\xi_{i,k,m}(t))d_{i}.\]
 By Proposition \ref{8new} the same structure constants define an associative 
  ${\mathbb C}\{t\}$-algebra on $\oplus _{i=1}^{n}{\mathbb C}\{t\}d_{i}$; reducing the constants modulo $t$  gives
 $d_{k} d_{m}=\sum_{i=1}^{n}\zeta _{i,k,m}d_{i}$ which is exactly the multiplication on $N$ from Proposition \ref{7new}.

 The fact that the image of the identity element in $N$ is the identity element in the obtained deformation follows from the fact that, by Proposition \ref{basisnew} (3), $q_{j}=1$; hence $f(q_{j})=1_{A'}$ and  $f(q_{j})f(q_{i})=f(q_{i})=f(q_{i})f(q_{j})$ for each $i\leq n$. This implies that $d_{i}\ast _{t}d_{j}=d_{i}=d_{j}\ast_{t}d_{i}$ (since the multiplication of $d_{k}\ast _{t}d_{m}$ uses the same constants $\xi _{i,k, m}(t), \zeta_{i,k, m}$ as the multiplication of $f(q_{k})f(q_{m})$). This implies that  
$d_{j}\ast _{0}d _{i}=d_{i}=d_{i}\ast _{0}d_{j}$ in $N$. 
 Since the identity element is unique in any ring, it follows that  $d_{j}$ is the identity in $N$ and its image is the identity in the obtained deformation algebra. This concludes the proof. 
\end{proof}

\section{Flat deformations to $A$}\label{Y}

 In this section we will show that $\mathcal N$ gives a flat deformation from $N$ to $A$.

\begin{theorem}\label{thm:iso1new}  
Let notation and assumptions be as in Subsection \ref{333}.  Let  $\mathcal N$ be a 
$\mathbb C\{t\}$-linear space  with  a basis $d_{1},\dots, d_n$   
 and with ${\mathbb C}\{t\}$-linear  associative multiplication given by  \[ d_{k}\ast _{t}d_{m}=\sum_{i=1}^{n}(\zeta _{i,k,m}+t\xi_{i,k,m}(t))d_{i},\]
  for all $k, m\leq n$. 
 Then this multiplication rule 
 gives a formal deformation of the algebra $N$ which is a flat deformation to $A$.
\end{theorem}
\begin{proof} The deformation is formal: setting $t=0$ in the multiplication rules for $\ast _{t}$ we obtain the multiplication  $\ast _{0}$.
  By Proposition \ref{Im}  $\mathcal N\cong \operatorname {Im} (f)$. By Proposition \ref{imp}  $ N\cong \operatorname {Im} (f)/t\cdot \operatorname {Im} (f)$.

We now show that $\operatorname {Im} (f)$ satisfies Definition \ref{Deformation} (so the obtained deformation is flat). By Propositions \ref{basisnew} (1) and \ref{basesnew} (1), $\operatorname {Im} (f)$ is a  
free ${\mathbb C}\{t\}$-module with basis $f(q_i{})$. By Proposition 
\ref{imp} $ \operatorname {Im} (f)/t\cdot \operatorname {Im} (f)$ is isomorphic to $N$.  After inverting $t$, ${\operatorname {Im} (f)}[t^{-1}]= A((t))$ (by Assumption $2$ from Subsection \ref{333} and a dimension comparison).  
 Hence 
 ${\operatorname {Im} (f)} \otimes _{{\mathbb C}\{t\}}{\mathbb C}\{\{t\}\}\cong A((t)) \otimes _{{\mathbb C}((t))}{\mathbb C}\{\{t\}\} \cong   A\otimes _{\mathbb C} {\mathbb C}\{\{t\}\}$.

\end{proof} 
\begin{corollary}\label{pppp} Let $n$ be a natural number. Let $N, A$ be unital  $\mathbb C$-algebras of dimension $n$. 
 If there is a unital ${\mathbb C}\{t\}$-algebra homomorphism $f$ satisfying the assumptions from Subsection \ref{333}  and such that $\operatorname {Im} (f)/t\cdot \operatorname {Im} (f)$ is isomorphic to $N$ as a unital  $\mathbb C$-algebra, then there is a flat deformation from $N$ to $A$.
\end{corollary}

\section{Answer to a Question from  \cite{DDS}}\label{DS}

 In this section we answer Question 0.1 from \cite{DDS}.
$ $
 
 We will now show that if there exists a flat deformation from a finite-dimensional algebra $N$ to an 
algebra $A$, then  there exists a flat deformation from $N$ to $A$ obtained by the  method from Subsection \ref{333}. Recall that by Definition $2.4$ of  \cite{JJAS} a deformation is \emph{polynomially split} if  it is of polynomial
            type and the associated ${\mathbb C}[t]$-algebra ${\mathcal N}_{{\mathbb C}[t]}$ satisfies   ${\mathcal N }_{{\mathbb C}[t]}[t^{-1}]\cong  A[t^{\pm 1}]$ 
(note  that in our notation $ A[t^{\pm 1}]={\tilde A}$). Recall that ${\mathcal N}_{{\mathbb C}[t]}=N[t]$ as a set. Moreover,
 ${\mathcal N}_{{\mathbb C}[t]}$ is a free ${\mathbb C}[t]$-module with an isomorphism of algebras ${\mathcal N}_{{\mathbb C}[t]}/t\cdot {\mathcal N}_{{\mathbb C}[t]}\cong N$ (see Remark $2.3$ of \cite{JJAS}). 
 Since a deformation of polynomial type  is a formal deformation, there is 
 a basis of the  $\mathbb C$-algebra $N$ with multiplication $\cdot $ and addition $+ $, and a new associative multiplication 
 $\circ _{t}$ of  $N\{t\}$ such that
 $a\circ_{t} b=a\cdot b+ u(a,b)$ where $u(a,b)\in t\cdot N\{t\}.$ This gives an associated algebra $\mathcal N$ which as a set is equal to $N\{t\}$ and which has multiplication $\circ _{t}$ (and has the same addition as $N\{t\}$).  
 Also ${\mathcal N}$ is a free ${\mathbb C}\{t\}$-module. We also have the associated algebra 
 ${\mathcal N}[t^{-1}]$ which as a set is equal to $N((t))$ (the Laurent series over $N$ in the variable $t$) and which has multiplication $\circ _{t}$ (and has the same addition as $N((t))$).
  For a deformation of polynomial type we also have an associated algebra ${ {\mathcal N}_{{\mathbb C}[t]}}$ which as a set is equal to $N[t]$   and has  the same addition as $N[t]$ and has the multiplication $\circ _{t}$ (where $N[t]$ is the polynomial ring over $N$ in variable $t$).
 Moreover, ${ {\mathcal N}_{{\mathbb C}[t]}}$ is a free ${\mathbb C}[t]$-module. We also have an associated algebra ${ {\mathcal N}_{{\mathbb C}[t]}}[t^{-1}]$ which as a set is equal to $N[t, t^{-1}]$   and has  the same addition as $N[t, t^{-1}]$ and has the multiplication $\circ _{t}$ (where $N[t, t^{-1}]$ is the ring of Laurent polynomials).
  If the deformation is polynomially split we have (by \cite{JJAS}) that ${ {\mathcal N}_{{\mathbb C}[t]}}[t^{-1}]$ is isomorphic to 
 $A[t, t^{-1}]$.  By tensoring ${{\mathcal N}_{{\mathbb C}[t]}}[t^{-1}]$ with ${\mathbb C}\{t\}[t^{-1}]$ (over ${\mathcal C}[t, t^{-1}]$) and tensoring 
 $A[t, t^{-1}]$ with  ${\mathbb C}\{t\}[t^{-1}]$ (over ${\mathcal C}[t, t^{-1}]$), we obtain that ${\mathcal N}[t^{-1}]$ is isomorphic to $A((t))=A'$.  
  Let ${d}_{1}', \ldots , {d}_{n}'$ be a basis of $N$, and let $d_{1}'=1_{N}$. 
 Let $d_{i}$ be the image of ${d}_{i}'$ in $
{\mathcal N}_{{\mathbb C}[t]}\subseteq {\mathcal N}$. 
 By Definition \ref{Deformation}, the image of $1_N$ is
the identity of the deformation. The polynomially split
deformation obtained from \cite{JJAS} has the same
convention. Hence $d_{1}=1_{\mathcal N}$. Then,  
  as sets, $N={d}_{1}'{\mathbb C}+ \cdots +{d}_{n}'{\mathbb C}$,
 $ {\mathcal N}_{{\mathbb C}[t]}=d_{1}{\mathbb C}[t]+ \cdots +d_{n}{\mathbb C}[t]$, 
${\mathcal N}=d_{1}{\mathbb C}\{t\}+ \cdots +d_{n}{\mathbb C}\{t\}$, and  $d_{1}, \ldots , d_{n}$ are as above. As sets $N\subseteq {\mathcal N}$, and as ${\mathbb C}[t]$-algebras we have ${ {\mathcal N}_{{\mathbb C}[t]}}\subseteq {\mathcal N}$. 

\begin{theorem}\label{25} Let notation be as above in this section. Let $N, A$ be finite-dimensional unital $\mathbb C$-algebras of dimension $n$. Suppose that there is a flat deformation from $N$ to $A$ (in the sense of Definition \ref{Deformation}). Then there is a formal deformation of $N$ which is a 
 flat deformation from $N$ to $A$,  and is obtained by the  method from Subsection \ref{333}. 
\end{theorem}
\begin{proof} Suppose that we have a flat deformation from $N$ to $A$, where $N, A$ are finite-dimensional algebras. 
 By Theorem $2.20$ of \cite{JJAS},  specifically the
implication $ (1)\rightarrow  (3)$, if there is a flat deformation from $N$ to $A$ then there is a polynomially split  flat deformation  from $N$ to $A$ which we denote by $ {\mathcal N}_{{\mathbb C}[t]}$ (which as a set is equal to $N[t]$). 
  Observe that $ {\mathcal N}_{{\mathbb C}[t]}$ is polynomially split, so there exists an isomorphism $h:{  {\mathcal N}_{{\mathbb C}[t]}}[t^{-1}]\rightarrow A[t^{\pm 1}]$  of ${\mathbb C}[t, t^{-1}]$-algebras. We also have  $A[t^{\pm 1}]\subseteq A\{t\}[t^{-1}]$. Since $d_{1}=1_{\mathcal N}$ and $h$ is unital, $h(d_{1})=1_{A[t^{\pm 1}]}$. We also have  $A[t^{\pm 1}]\subseteq A\{t\}[t^{-1}]$ and in this natural embedding we have $1_{A[t^{\pm 1}]  }=1_{  A\{t\}[t^{-1}] }$, so  we have 
 $h(d_{1})=1_{A\{t\}[t^{-1}]}\in  A\{t\}[t^{-1}]$.
 Define a homomorphism  $f: {\mathbb C}\langle x_{1}, \ldots , x_{n-1}\rangle \{t\}\rightarrow A'$ of  unital ${\mathbb C}\{t\}$-algebras by 
\[f(1)=h(d_{1})=1_{A\{t\}[t^{-1}]}\] and  \[f(x_{i})=h(d_{i+1})\] for  $i=1, \ldots , n-1$.  Then $f$ is a well-defined  homomorphism of unital ${\mathbb C}\{t\}$-algebras (for $n=1$ we assume the convention ${\mathbb C}\{\emptyset \}\{t\}={\mathbb C}\{t\}$, in this case $A$ and $N$ are one-dimensional unital algebras so they are isomorphic, hence the results hold in this case; so we can assume that $n>1$). 

  We will now show that  $f$  satisfies the assumptions of Subsection \ref{333} for $\rho =n-1$.  We use the notation from the beginning of this section. 

 {\bf Verification of Assumption $1$.}  Notice that by our definition $f(1)=1_{ A\{t\} [ t^{-1}] }$ and 
 $f(x_{i})=h(d_{i+1})\in A[t, t^{-1}]$, so Assumption $1$ holds. 

{\bf Verification of Assumption $2$.}  Since $h$ is an isomorphism and $ {\mathcal N}_{{\mathbb C}[t]}=d_{1}{\mathbb C}[t]+ \cdots + d_{n}{\mathbb C}[t]$ we get
  \[A[t, t^{-1}]=h({{\mathcal N}_{{\mathbb C}[t]}}[t^{-1}])= h(d_{1}){\mathbb C }[t, t^{-1}]+\ldots + h(d_{n}){\mathbb C }[t, t^{-1}], \]  
so $A[t, t^{-1}]\subseteq {\operatorname {Im} (f) }[t^{-1}]$ (by ${\operatorname {Im} (f) }[t^{-1}]$ 
 we mean  the subset of $A\{t\}[t^{-1}]$ consisting of finite sums of elements of the form $rt^{i}$ where $i\in \mathbb Z$ and $r\in \operatorname {Im} (f)$). Note that the existence of $h$ implies that ${ {\mathcal N}_{{\mathbb C}[t]}}[t^{-1}]$ has an identity element $1_{ { {\mathcal N}_{{\mathbb C}[t]}}[t^{-1}]}$, since $A[t, t^{-1}]$ has an identity element $1_{A[t, t^{-1}]}$. 
 Since each element of ${{\mathcal N}_{{\mathbb C}[t]}}[t^{-1}]$  is a ${\mathbb C}[t, t^{-1}]$-linear combination of elements $d_{1}, \ldots , d_{n}$, then $1_{ {\mathcal N}_{{\mathbb C}[t] }}[t^{-1}]$ is also such a linear combination.  It follows that 
 $1_{A[t, t^{-1}]}$ is a  ${\mathbb C}[t, t^{-1}]$-linear combination of elements $h(d_{1}), \ldots , h(d_{n})$. 
 Now, $A[t, t^{-1}]\subseteq  \operatorname {Im} (f)[t^{-1}]$ implies $A'=A\{t\}[t^{-1}]\subseteq  \operatorname {Im} (f)[t^{-1}]$. 
 Since $\operatorname {Im} (f)\subseteq A'$, we get 
 $\operatorname {Im} (f) [t^{-1}]=A'$. So for every $r\in A'$ (including the case when  $r=1_{A'}$) there is $m(r)\geq 0$ such that \[r\cdot t^{m(r)}\in \operatorname {Im} (f).\] Recall that $A$ is an $n$-dimensional $\mathbb C$-algebra. Choose a basis  ${b}_{1}, \ldots , {b}_{n}$ of $A$. Then ${b}_{1}t^{\alpha }, \ldots , {b}_{n}t^{\alpha } $ are linearly independent over ${\mathbb C}\{t\}$. By the above,  choose  $\alpha $ such that  ${b}_{1}t^{\alpha }, \ldots , {b}_{n}t^{\alpha }\in \operatorname {Im} (f)$.
 It follows that there are $c_{i}'\in {\mathbb C}\langle x_{1}, \ldots , x_{\rho }\rangle \{t\}$ such that $f(c_{i}')={b}_{i}t^{\alpha }$.  It follows that $f(c_{1}'), \ldots , f(c_{n}')$ are linearly independent over ${\mathbb C}\{t\}$.  Since each $c_{i}'$ is a finite ${\mathbb C}\{t\}$-linear combination of words, the ${\mathbb C}\{t\}$-span of the
word images has rank $n$; therefore some $n$ word images are ${\mathbb C}\{t\}$-linearly
independent.
 Consequently,  there are  $c_{1}, \ldots , c_{n}\in {\mathbb C}\langle x_{1}, \ldots , x_{\rho }\rangle $ such that 
 $f(c_{1}), \ldots , f(c_{n})$ are linearly independent over ${\mathbb C}\{t\}$. 
Therefore, Assumption $2$ from Subsection \ref{333} holds. 

{\bf Verification of Assumption $3$.}  By the definition of $f$ we have that 
  $\operatorname {Im} (f)$ equals  the  ${\mathbb C}\{t\}$-subalgebra  of $A'$ generated by the elements 
$f(x_{i})=h(d_{i+1})$ (recall that $1_{A'}\in \operatorname {Im} (f)$).   
 Since $d_{i}\in {{\mathcal N}_{{\mathbb C}[t]}}$,  they satisfy 
 relations $d_{k}\circ _{t} d_{m}=\sum_{i=1}^{n} (\zeta _{i, k, m}+t\xi _{i, k, m}(t))d_{i}$  (where $\xi _{i, k, m}(t)$ are polynomials since the deformation is of polynomial type) so in $A'=A((t))$ we have relations 
\[h(d_{k}) h( d_{m})=\sum_{i=1}^{n}( \zeta _{i,k,m}+t\xi _{i, k, m}(t))h(d_{i}).\] 
 Therefore, since $h(d_{1})=1_{A\{t\}[t^{-1}]}$ and for $1<i\leq n$, $h(d_{i})=f(x_{i-1})\in \operatorname {Im} (f)$, we have  \[\operatorname {Im} (f)=\sum_{i=1}^{n}h(d_{i}){\mathbb C}\{t\},\]
  since any product of elements $h(d_{i})$ belongs to $\sum_{i=1}^{n}h(d_{i}){\mathbb C}\{t\}$ by the above relations.  Observe that there exists an integer $\gamma $ such that $h(d_{1}), \ldots , h(d_{n})\in t^{-\gamma }A\{t\}$.   By the above equations  a product of 
$h(d_{i})$’s can be reduced  to the ${\mathbb C}\{t\}$-span
of finitely many $h(d_{i})$’s. 
 Therefore, $\operatorname {Im} (f)\subseteq t^{-\gamma }A\{t\}$. Consequently, Assumption $3$ of Subsection \ref{333} holds. 

 By Proposition \ref{imp}
 it remains to show that \[\operatorname {Im} (f)/t\cdot \operatorname {Im} (f)\] is isomorphic to $N$.  
  By the proof that Assumption $3$ holds,  $\operatorname {Im} (f)={\mathbb C}\{t\}h(d_{1})+\ldots +{\mathbb C}\{t\}h(d_{n})$.

 Since $\mathcal N$ is a formal deformation of $N$, it follows that $\mathcal N/t\cdot {\mathcal N}$ is isomorphic to $N$. Since $N$ has dimension $n$ as a $\mathbb C$-algebra, \[\mathcal N=d_{1}{\mathbb C}\{t\}+\ldots +d_{n}{\mathbb C}\{t\},\] where $d_{i}$ denotes the image of $d_{i}'$ in
${\mathcal  N}_{{\mathbb C}[t]} \subseteq {\mathcal N}$. Since $d_{1}, \ldots ,  d_{n}$ 
form a ${\mathbb C}\{t\}$-basis of $\mathcal N$, their
residue classes $d_{1}+t\cdot  {\mathcal N}, \ldots , d_{n}+t\cdot  {\mathcal N}$
 form a $\mathbb C$-basis of $\mathcal N/t\cdot {\mathcal N}$. 
 Let  $g:  {\mathcal N}/t\cdot {\mathcal N}\rightarrow N$ be an  isomorphism of $\mathbb C$-algebras. It follows that $g(d_{1}+t{\mathcal N}), \ldots , g(d_{n}+t{\mathcal N})$ is a basis of $N$ as a linear space over $\mathbb C$.
  
Consider the ${\mathbb C}$-linear map $q: \operatorname {Im} (f) \rightarrow     N $ given for $\alpha _{i} (t)\in {\mathbb C}\{t\}$  by   
  \[q(\sum_{i=1}^{n} \alpha_{i} (t)\cdot h(d_{i}))=\sum_{i=1}^{n}\alpha _{i} (0)\cdot g(d_{i}+t{\mathcal N}).\]  

 Since $h$ is an isomorphism, it follows that $h(d_{1}), \ldots , h(d_{n})$ are linearly independent over ${\mathbb C}[t, t^{-1}]$. Let $a_{1}, \ldots , a_{n}$ be a ${\mathbb C}$-basis of $A$. Then there is a non-zero  polynomial $p(t)\in {\mathbb C}[t]$ such that $a_{1}p(t), \ldots , a_{n}p(t)\in h(d_{1}){\mathbb C}[t, t^{-1}]+\ldots +h(d_{n}){\mathbb C}[t, t^{-1}]$. Therefore, $a_{1}, \ldots , a_{n}\in h(d_{1}){\mathbb C}\{ t\}[t^{-1}]+\ldots +h(d_{n}){\mathbb C}\{t\}[ t^{-1}]$.
Therefore,    $h(d_{1}), \ldots , h(d_{n})$  span $A\{t\}[t^{-1}]$ as a ${\mathbb C}\{t\}[t^{-1}]$-vector space. Since $A\{t\}[t^{-1}]$ has dimension $n$ as a ${\mathbb C}\{t\}[t^{-1}]$-vector space (since $A$ has dimension $n$) it follows that    $h(d_{1}), \ldots , h(d_{n})$
are linearly independent over ${\mathbb C}\{t\}$. Therefore,  a presentation of an element of $Im(f)$ as a sum $\sum_{i=1}^{n}\alpha_{i} (t)\cdot h(d_{i})$ is unique, so $q$ is well-defined. 

 The map $q$ is additive. To prove that $q$ is a ${\mathbb C}$-algebra homomorphism, it remains to verify
multiplicativity.  Let \[a=\sum_{i=1}^{n}\alpha_{i} (t)\cdot h(d_{i}), b=\sum_{j=1}^{n}\alpha_{j}' (t)\cdot h(d_{j}),\] for some $\alpha _{i}(t), \alpha _{j}'(t)\in {\mathbb C}\{t\}$. 
 Then \[ab=\sum_{i=1}^{n}\sum_{j=1}^{n}\alpha_{i} (t)\alpha _{j}' (t)\cdot h(d_{i}\circ _{t}d_{j})=\sum_{i=1}^{n}\sum_{j=1}^{n}\sum_{l=1}^{n}\alpha_{i} (t)\alpha _{j}' (t)  (\zeta _{l,i,j}+t\xi _{l,i,j}(t))h(d_{l}).\]
 Therefore, \[q(ab)= \sum_{i=1}^{n}\sum_{j=1}^{n}\sum_{l=1}^{n}\alpha_{i} (0)\alpha _{j}' (0)\cdot  (\zeta _{l,i,j}+0\cdot \xi _{l,i,j}(0))g(d_{l}+t{\mathcal N}).\]

 On the other hand, 
\[q(a)q(b)= \sum_{i=1}^{n}\sum_{j=1}^{n}\alpha_{i} (0)\alpha _{j}'(0)  g(d_{i}\circ _{t} d_{j}+t\cdot {\mathcal N}).\]
 Recall that $g$ is a homomorphism of algebras, and $d_{i}\circ _{t}d_{j}-\sum_{l=1}^{n}\zeta _{l,i,j}d_{l}\in t\cdot {\mathcal N}$. 

 Therefore, \[q(a)q(b) =  \sum_{i=1}^{n}\sum_{j=1}^{n}\sum_{l=1}^{n}\alpha_{i} (0)\alpha _{j}' (0)  \zeta _{l,i,j}g(d_{l}+t{\mathcal N}) .\]

 Therefore $q(ab)=q(a)q(b)$, which proves that $q$ is a homomorphism. 
 
Observe that if $r=\sum_{i=1}^{n}\alpha_{i} (t)h(d_{i})$ is in the kernel of $q$ then 
\[0_{N}=q(r)=\sum_{i=1}^{n}\alpha_{i} (0)g(d_{i}+t{\mathcal N}),\] and since 
$g(d_{1}+t{\mathcal N}), \ldots , g(d_{n}+t{\mathcal N})$ is a basis of $N$ as a linear space over $\mathbb C$, $\alpha _{i}(0)=0$, so $\alpha _{i}(t)\in t\mathbb C\{t\}$ for each $i$. Consequently, $r$ is in the kernel of $q$ if and only if $r\in t( {\mathbb C}\{t\}h(d_{1})+\ldots + {\mathbb C}\{t\}h(d_{n})) =t\cdot \operatorname {Im} (f)$. 
 Observe that $q$ is onto because the elements $g(d_{i}+tN)$  for $1 \leq  i \leq n$ form a $\mathbb C$-basis of $N$.

 It follows that 
$\operatorname {Im} (f)/ker (q)=\operatorname {Im} (f)/t\cdot \operatorname {Im} (f)$, and by the First Isomorphism Theorem for $\mathbb C$-algebras it is isomorphic to $ N$. Therefore,  $\operatorname {Im} (f)/t\cdot \operatorname {Im} (f)$ is isomorphic to $N$, as a $\mathbb C$-algebra.  
\end{proof}

 By Corollary \ref{pppp} and Theorem \ref{25} we obtain: 
\begin{corollary} \label{corollary} Let $n$ be a natural number. Let $N, A$ be unital $\mathbb C$-algebras of dimension $n$.
 The following statements are equivalent:
\begin{itemize}
\item There is a flat deformation from $N$ to $A$  (in the sense of Definition \ref{Deformation}).
\item  There exists a unital ${\mathbb C}\{t\}$-algebra  homomorphism $f$ satisfying  the assumptions from Subsection \ref{333} and such that $\operatorname {Im} (f)/t\cdot \operatorname {Im} (f)$ is isomorphic as a  unital $\mathbb C$-algebra to $N$. 
\end{itemize}
\end{corollary}

 \section{Generalizations}\label{generalisations}

We give a slight modification of the method from Subsection \ref{333}, in a special case where $A$ is semisimple. This method is easier to apply in practice.


 \subsection{A slight modification}\label{71} 

In this subsection we modify the method from Subsection \ref{333} to obtain another  method for constructing flat deformations from an algebra $N$ to a semisimple algebra $A$. 

 In a follow-up paper we will use this method  to construct deformations of some contraction algebras of type $E$. 
This method is easier to use in practice than the method from Subsection \ref{333}.

  By \[F={\mathbb C}\{t\}[t^{-1}]\] we denote the field of fractions of ${\mathbb C}\{t\}$ (the field of Laurent series). 
 Recall that ${\mathbb C}\{t\}$ is the ring of power series with coefficients from the field of complex numbers $\mathbb C$. 
 
    Let $k', n_{1}, \ldots , n_{k'}$ be positive integers. For each $i\leq k'$, we fix a polynomial  $g_{i}(x)\in {\mathbb C}[t][x]$ of degree $d_{i}>0$ ($g_{i}$ is a polynomial in the variable $x$ with coefficients in  ${\mathbb C}[t]\subseteq {F}$). We assume that the coefficient of the highest  power of $x$ in $g_{i}(x)$ is $1$.  We also fix $u_{i}=x+I_{i}\in F[x]/I_{i}$ where $I_{i}$ is the ideal generated by $g_{i}(x)$ in the polynomial ring $F[x]$. Observe that $g_{i}(u_{i})=0$  in $F[x]/I_{i}$.

 We assume  that each   $g_{i}(x)$ has no multiple roots in the algebraic closure of the field $\mathbb F$ (so   
 the  polynomial $g_{i}(x)$ and its derivative  $g_{i}'(x)$  have no nonconstant common divisors). We sometimes write $g_{i}(x)(t)$  to emphasize the dependence of its coefficients on $t$.

$ $
\begin{definition}\label{909}
 We denote \[A'=\oplus _{i=1}^{k'}M_{n_{i}}({\mathbb C}\{t\}[t^{-1}][u_{i}]),\]
 where ${\mathbb C}\{t\}[t^{-1}][u_{i}]$ consists of  finite sums of elements of the form  $c(t)u_{i}^{k}t^{j}$ for $k=0,1,\ldots $, 
$j\in {\mathbb Z}$, $c(t)\in {\mathbb C}\{t\}$. Similarly, ${\mathbb C}[t, t^{-1}][u_{i}]$ consists of  finite sums of terms   of the form $ cu_{i}^{k}t^{j}$ for $k=0,1,\ldots $, 
$j\in \mathbb Z$, where $c$ is in ${\mathbb C}$.
 We define  $M$ to be the following $\mathbb C$-linear subspace of $A'$:  \[{ M}=\oplus_{i=1}^{k'}M_{n_{i}}({\mathbb C}+{\mathbb C}u_{i}+\cdots +{\mathbb C}u_{i}^{d_{i}-1}).\] 
 Observe that $M[t, t^{-1}]$ is an algebra, where $M[t, t^{-1}]$ consists of finite sums of elements $mt^{i}$ for $m\in M$, $i\in \mathbb Z$ (closure of multiplication  follows because $g_i$  is monic with coefficients
in ${\mathbb C}[t]$, so every power $u_{i}^m$ reduces to a ${\mathbb C}[t]$-linear
combination of $1, u_{i},\ldots ,u_{i}^{d_{i-1}}$, and powers of $t$ are
allowed in $M[t,t^{-1}]$).  Denote ${\tilde A}=M[t, t^{-1}]$. 
 We denote $n=\sum_{i=1}^{k'}n_{i}^{2}d_{i}$. 
 We define $M\{t\}$ to be the following ${\mathbb C}\{t\}$-linear subspace of $A'$:  \[{ M}\{t\}=\oplus_{i=1}^{k'}M_{n_{i}}({\mathbb C}\{t\}+{\mathbb C}\{t\}u_{i}+\cdots +{\mathbb C}\{t\}u_{i}^{d_{i}-1}).\] 
 Denote  by $M\{t\}[ t^{-1}]$ the space consisting of finite sums of elements $mt^{i}$ for $m\in M\{t\}$, $i\in \mathbb Z.$ Observe that $A'=M\{t\}[t^{-1}]$. 
\end{definition}
\subsection { The method} \label{73}

 Let notation be as in Subsection \ref{71}. Let $\rho$ be a natural number and let  ${\mathbb C}\langle x_{1}, \ldots , x_{\rho }\rangle \{t\}$ be defined as in Notation \ref{4444}.
 We first describe the {\bf assumptions of our method}.

\begin{enumerate}
 \item  Let 
 $f: {\mathbb C}\langle x_{1}, \ldots , x_{\rho }\rangle \{t\} \rightarrow A'$  be a homomorphism of ${\mathbb C}\{t\}$-algebras such that \[f(x_{j})\in M[t, t^{-1}],\] for $j=1, \ldots , \rho$, 
 where $M[t, t^{-1}]$ is as in Definition \ref{909}. 

  We also have $f(1)=1_{A'}$, where $1_{A'}$ is the identity element of $A'$.
 
{\em Remark.}  {\em Notice that if  all $g_{i}(x)$ have degree one then all  $u_{i}\in {\mathbb C}[t]$ (so $u_{i}$ can be omitted) and  we would obtain the method from Subsection \ref{333}}. 

\item Denote \[{ A}=\oplus_{i=1}^{k'} M_{n_{i}}({\mathbb C})^{\oplus d_{i}}.\] 
 Then $A$ has dimension $n=\sum_{i=1}^{k'}n_{i}^{2}\cdot d_{i}.$

\item We suppose that there are $m_{1}, \ldots , m_{n}\in {\mathbb C}\langle x_{1}, \ldots , x_{\rho }\rangle $ such that 
$f(m_{1}), \ldots , f(m_{n})$ are linearly independent over $\mathbb C \{t\}$. 
\item We also assume that there is a natural number $\gamma $ such that
 \[\operatorname {Im} (f)\subseteq t^{-\gamma} M\{t\},\]
 where 
 $M\{t\}$ is as in Definition \ref{909}. 
 Recall that  $M$ is the following linear subspace of $A'$:  \[{ M}=\oplus_{i=1}^{k'}M_{n_{i}}({\mathbb C}+{\mathbb C}u_{i}+\cdots +{\mathbb C}u_{i}^{d_{i}-1}).\] 
\end{enumerate}
{\em If the above properties $1$--$4$ are satisfied, then there is a flat deformation from the algebra $\operatorname {Im} (f)/t\cdot \operatorname {Im} (f)$ to $A$ (see Corollary \ref{final})}.

\subsection{Validity of the method in the new notation }

   Observe first that  all the proofs from Sections \ref{A} to \ref{Y}
  work if we take $A, A', {\tilde A}, { M}$ as in Subsection \ref{73} (instead of taking them from Subsection \ref{333}).
 Recall that in Subsection \ref{333} we have $A((t))=A'$, $A[t,t^{-1}]={\tilde A}$, $A=M$.  So Sections \ref{A}-\ref{Y} carry over after replacing $A$ by $M$, replacing ${\tilde A}=A[t, t^{-1}]$ by ${\tilde A}=M[t, t^{-1}]$, and 
 replacing $A'=A((t))$ by $A'=M\{t\}[t^{-1}]$.

 Note that  Sections \ref{B}-\ref{Z}  use only properties of elements $q_{i}$, and symbols   
$A, A', {\tilde A}, { M} $ do not  appear in these sections (except that in Section \ref{6} the assumption $f(1)=1_{A'}$ is used). 
 
In the
present notation, Section \ref{A} is used without the assumption \[M=A, \tilde A=A[t,t^{-1}], A'=A((t)).\] The arguments of Section \ref{A} remain valid without the standing assumption
 $M=A, \tilde A=A[t,t^{-1}], A'=A((t))$, which appears  in the first line of Section \ref{A}.
 The structural properties which are being used are the following:
\begin{itemize}
\item  $M$ is an $n$-dimensional $\mathbb C$-vector space containing the identity of $A'$;
\item $M\{t\}$ is a free ${\mathbb C}\{t\}$-module of rank $n$;
\item $A'=M\{t\}[t^{-1}]$;
\item $M[t,t^{-1}]$ is an algebra containing the generator images;
\item $\operatorname {Im} (f)$ is bounded below in $t$-valuation.
\end{itemize}
 Therefore the leading-term construction in Proposition \ref{basisnew} and the unit argument still work, even though $M$ itself is only a linear subspace, not an algebra.

 In Section \ref{Y} the proof of Theorem \ref{thm:iso1new} is mostly the same. The only place that requires additional  work is the last line of that proof, namely the  fact that ${\operatorname {Im} (f)} \otimes _{{\mathbb C}\{t\}}{\mathbb C}\{\{t\}\}\cong   A\otimes _{\mathbb C} {\mathbb C}\{\{t\}\}$ (so $\operatorname {Im} (f)$ satisfies Definition \ref{Deformation}).  Observe first that ${\operatorname {Im} (f)}\otimes _{{\mathbb C}\{t\}}{\mathbb C}\{\{t\}\}\cong {\operatorname {Im} (f)}[t^{-1}]\otimes _{{\mathbb C}((t))}{\mathbb C}\{\{t\}\}=A'\otimes _{{\mathbb C}((t))}{\mathbb C}\{\{t\}\}=M\{t\}[t^{-1}]\otimes _{{\mathbb C}((t))}{\mathbb C}\{\{t\}\}\cong M\{t\}\otimes _{{\mathbb C}\{t\}}{\mathbb C}\{\{t\}\}$, since $\operatorname {Im} (f) [t^{-1}]=M\{t\}[t^{-1}]$ ($A'=M\{t\}[t^{-1}]$ and Assumption $4$ from Subsection \ref{73} gives containment after localization, and Assumption $3$ from Subsection \ref{73} gives full dimension).  This follows from the fact that $u_{i}$ is a root of the monic polynomial $g_{i}(x)\in {\mathbb C}\{t\}[x]$ which has no multiple roots, and so $({\mathbb C}\{t\}+{\mathbb C}\{t\}u_{i}+\cdots +{\mathbb C}\{t\}u_{i}^{d_{i}-1})\otimes _{{\mathbb C}\{t\}}{\mathbb C}\{\{t\}\}\cong {\mathbb C}\{\{t\}\}^{\oplus d_{i}}$. Therefore, $M_{n_{i}}({\mathbb C}\{t\}+{\mathbb C}\{t\}u_{i}+\cdots +{\mathbb C}\{t\}u_{i}^{d_{i}-1})\otimes _{{\mathbb C}\{t\}}{\mathbb C}\{\{t\}\}\cong M_{n_{i}}({\mathbb C}\{\{t\}\})^{\oplus d_{i}}.$  Therefore, $M\{t\}\otimes _{{\mathbb C}\{t\}}{\mathbb C}\{\{t\}\}\cong A\otimes _{\mathbb C}{\mathbb C}\{\{t\}\}$. 
 We now give a  proof that \[({\mathbb C}\{t\}+{\mathbb C}\{t\}u_{i}+\cdots +{\mathbb C}\{t\}u_{i}^{d_{i}-1})\otimes _{{\mathbb C}\{t\}}{\mathbb C}\{\{t\}\}\cong {\mathbb C}\{\{t\}\}^{\oplus d_{i}}.\]
 Observe that $({\mathbb C}\{t\}+{\mathbb C}\{t\}u_{i}+\cdots +{\mathbb C}\{t\}u_{i}^{d_{i}-1})\cong {\mathbb C}\{t\}[x]/\langle g_{i}(x)\rangle $, and 
\[ {\mathbb C}\{t\}[x]/\langle g_{i}(x)\rangle\otimes  _{{\mathbb C}\{t\}}{\mathbb C}\{\{t\}\}\cong {\mathbb C}\{\{t\}\}[x]/\langle g_{i}(x)\rangle \cong {\mathbb C}\{\{t\}\}^{\oplus d_{i}},\]
where the last isomorphism is the Chinese remainder theorem, since $g_{i}$ has $d_{i}$ distinct roots in
the algebraically closed field ${\mathbb C}\{\{t\}\}$.

As a corollary we obtain:
\begin{corollary}\label{final}Let $n$ be a natural number. Let $N, A$ be unital $\mathbb C$-algebras of dimension $n$, with $A$ semisimple. Choose a
Wedderburn decomposition $A\cong \oplus_{i=1}^{k'} M_{n_{i}}({\mathbb C})^{\oplus d_{i}}$ where $d_{i}, n_{i}, k'$ are such that  $n=\sum_{i=1}^{k'}n_{i}^{2}\cdot d_{i}.$

 If there is a unital ${\mathbb C}\{t\}$-algebra homomorphism $f$ satisfying the  assumptions from Subsection \ref{73}  and such that $Im(f)/t\cdot Im(f)$ is isomorphic to $N$ as a $\mathbb C$-algebra, then there is a flat deformation from $N$ to $A$.
\end{corollary}

\subsection{More examples}
 In this section we give an example illustrating the generalized method.

{\bf Example.} Let $r\geq 2$ and put 
$ F={\mathbb C}\{t\}[t^{-1}].$ 
 Consider the polynomial 
\[g(X)=X^{r}-t\in F[X],\]
and let 
\[B=F[X]/\langle X^{r}-t\rangle, u=X+\langle X^{r}-t\rangle\in B.\]
 Define a ${\mathbb C}\{t\}$-algebra homomorphism
\[f: {\mathbb C}\langle x\rangle \{t\} \longrightarrow B, f(x)=u.\]
 Then, $u^{r}=t.$
 Hence the image of $f$ is 
$\operatorname {Im} (f)={\mathbb C}\{t\}[u]={\mathbb C}\{t\}\cdot 1+ {\mathbb C}\{t\}\cdot u+\cdots +
{\mathbb C}\{t\}\cdot u^{r-1}.$
This is a free ${\mathbb C}\{t\}$-module of rank $r$. Therefore, the assumptions of Subsection \ref{73} are satisfied. We recall Lemma \ref{condition2}, which also holds in our generalized setting and is  applied here with $\rho =1$, since we have one generator $x$. Let $m$ be a natural number. Let $s_{1}, \ldots, s_{m}\in {\mathbb C}\langle x_{1}, \ldots , x_{\rho }\rangle $ be such that $f(s_{i})\in t\cdot \operatorname {Im} (f)$.
 Let $S$ be the ideal generated by $s_{1}, \ldots , s_{m}$ in  ${\mathbb C}\langle x_{1}, \ldots , x_{\rho }\rangle $.   
 Suppose that the dimension of the algebra 
 $ {\mathbb C}\langle x_{1}, \ldots , x_{\rho }\rangle /S$ is at most $n$. Then 
 $ {\mathbb C}\langle x_{1}, \ldots , x_{\rho }\rangle /S$ is isomorphic to $N$.

  In our example, since there is one
generator, ${\mathbb C}\langle x\rangle = {\mathbb C}[x]$.
  Recall that $n$ is the dimension of $M$ and in our case  $M={\mathbb C}\cdot 1+ {\mathbb C}\cdot u+\cdots +
{\mathbb C}\cdot u^{r-1}$ so $n=r$.  
 Let $s_{1}=x^{r}$. Then \[f(s_{1})=u^{r}=t=t\cdot f(1),\] so \[f(s_{1})\in t\cdot \operatorname {Im} (f).\] Let $S$ be the ideal generated by $s_{1}$; then 
 $ {\mathbb C}\langle x\rangle /S$ has dimension not exceeding $r$.
By Lemma  \ref{condition2} (applied to our generalized construction), 
 $ {\mathbb C}\langle x\rangle /S$ is isomorphic to $N$. 
 Hence our method gives  a flat deformation from 
 $ {\mathbb C}\langle x\rangle /S= {\mathbb C}\langle x\rangle /\langle x^{r}\rangle $ to ${\mathbb C}^{\oplus r}$. 
  Some other applications will be illustrated in a follow-up paper for some  contraction algebras  of type $E$.

\end{document}